\documentclass[11pt]{article}
\usepackage[margin=1in]{geometry} 
\usepackage{amssymb,amsfonts,amsmath,bbm,mathrsfs,stmaryrd,mathtools,mathabx}
\usepackage{xcolor}
\usepackage{url}

\usepackage{enumerate}

\usepackage{hyperref}
\hypersetup{colorlinks,
             linkcolor=black!75!red,
             citecolor=blue,
             pdftitle={},
             pdfproducer={pdfLaTeX},
             pdfpagemode=None,
             bookmarksopen=true
             bookmarksnumbered=true}

\usepackage{tikz}
\usetikzlibrary{arrows,calc,decorations.pathreplacing,decorations.markings,intersections,shapes.geometric,through,fit,shapes.symbols,positioning,decorations.pathmorphing}

\usepackage{braket}

\usepackage[amsmath,thmmarks,hyperref]{ntheorem}
\usepackage{cleveref}

\creflabelformat{enumi}{#2(#1)#3}

\crefname{section}{Section}{Sections}
\crefformat{section}{#2Section~#1#3} 
\Crefformat{section}{#2Section~#1#3} 

\crefname{subsection}{\S}{\S\S}
\AtBeginDocument{%
  \crefformat{subsection}{#2\S#1#3}%
  \Crefformat{subsection}{#2\S#1#3}%
}

\crefname{subsubsection}{\S}{\S\S}
\AtBeginDocument{%
  \crefformat{subsubsection}{#2\S#1#3}%
  \Crefformat{subsubsection}{#2\S#1#3}%
}

\theoremstyle{plain}

\newtheorem{lemma}{Lemma}[section]
\newtheorem{proposition}[lemma]{Proposition}
\newtheorem{corollary}[lemma]{Corollary}
\newtheorem{theorem}[lemma]{Theorem}

\theoremstyle{nonumberplain}
\newtheorem{theoremN}{Theorem}

\theoremstyle{plain}
\theorembodyfont{\upshape}
\theoremsymbol{\ensuremath{\blacklozenge}}

\newtheorem{definition}[lemma]{Definition}
\newtheorem{example}[lemma]{Example}

\newtheorem{remark}[lemma]{Remark}
\newtheorem{remarks}[lemma]{Remarks}

\newtheorem{construction}[lemma]{Construction}

\crefname{definition}{definition}{definitions}
\crefformat{definition}{#2definition~#1#3} 
\Crefformat{definition}{#2Definition~#1#3} 

\crefname{ex}{example}{examples}
\crefformat{example}{#2example~#1#3} 
\Crefformat{example}{#2Example~#1#3} 

\crefname{exs}{example}{examples}
\crefformat{examples}{#2example~#1#3} 
\Crefformat{examples}{#2Example~#1#3} 

\crefname{remark}{remark}{remarks}
\crefformat{remark}{#2remark~#1#3} 
\Crefformat{remark}{#2Remark~#1#3} 

\crefname{remarks}{remark}{remarks}
\crefformat{remarks}{#2remark~#1#3} 
\Crefformat{remarks}{#2Remark~#1#3} 

\crefname{convention}{convention}{conventions}
\crefformat{convention}{#2convention~#1#3} 
\Crefformat{convention}{#2Convention~#1#3} 

\crefname{notation}{notation}{notations}
\crefformat{notation}{#2notation~#1#3} 
\Crefformat{notation}{#2Notation~#1#3} 

\crefname{table}{table}{tables}
\crefformat{table}{#2table~#1#3} 
\Crefformat{table}{#2Table~#1#3}

\crefname{lemma}{lemma}{lemmas}
\crefformat{lemma}{#2lemma~#1#3} 
\Crefformat{lemma}{#2Lemma~#1#3} 

\crefname{proposition}{proposition}{propositions}
\crefformat{proposition}{#2proposition~#1#3} 
\Crefformat{proposition}{#2Proposition~#1#3} 

\crefname{corollary}{corollary}{corollaries}
\crefformat{corollary}{#2corollary~#1#3} 
\Crefformat{corollary}{#2Corollary~#1#3} 

\crefname{theorem}{theorem}{theorems}
\crefformat{theorem}{#2theorem~#1#3} 
\Crefformat{theorem}{#2Theorem~#1#3} 

\crefname{enumi}{}{}
\crefformat{enumi}{(#2#1#3)}
\Crefformat{enumi}{(#2#1#3)}

\crefname{assumption}{assumption}{Assumptions}
\crefformat{assumption}{#2assumption~#1#3} 
\Crefformat{assumption}{#2Assumption~#1#3} 

\crefname{construction}{construction}{Constructions}
\crefformat{construction}{#2construction~#1#3} 
\Crefformat{construction}{#2Construction~#1#3} 

\crefname{equation}{}{}
\crefformat{equation}{(#2#1#3)} 
\Crefformat{equation}{(#2#1#3)}

\numberwithin{equation}{section}

\theoremstyle{nonumberplain}
\theoremsymbol{\ensuremath{\blacksquare}}

\newtheorem{proof}{Proof}
\newcommand\pf[1]{\newtheorem{#1}{Proof of \Cref{#1}}}

\newcommand\bC{{\mathbb C}}

\newcommand\bQ{{\mathbb Q}}
\newcommand\bR{{\mathbb R}}
\newcommand\bS{{\mathbb S}}
\newcommand\bT{{\mathbb T}}

\newcommand\bZ{{\mathbb Z}}

\newcommand\cC{{\mathcal C}}
\newcommand\cD{{\mathcal D}}

\newcommand\cF{{\mathcal F}}

\newcommand\cH{{\mathcal H}}

\newcommand\cO{{\mathcal O}}
\newcommand\cP{{\mathcal P}}

\newcommand\cS{{\mathcal S}}

\newcommand\1{{\bf 1}}

\DeclareMathOperator{\id}{id}
\DeclareMathOperator{\rk}{rank}
\DeclareMathOperator{\End}{\mathrm{End}}
\DeclareMathOperator{\Hom}{\mathrm{Hom}}
\DeclareMathOperator{\Aut}{\mathrm{Aut}}

\newcommand{\cat}[1]{\textsc{#1}}

\newcommand{\qedhere}{\mbox{}\hfill\ensuremath{\blacksquare}}

\title{Non-degeneracy results for (multi-)pushouts of compact groups}
\author{Alexandru Chirvasitu}

\begin{document}

\date{}

\newcommand{\Addresses}{{
  \bigskip
  \footnotesize

  \textsc{Department of Mathematics, University at Buffalo}
  \par\nopagebreak
  \textsc{Buffalo, NY 14260-2900, USA}  
  \par\nopagebreak
  \textit{E-mail address}: \texttt{achirvas@buffalo.edu}


}}

\maketitle

\begin{abstract}
  We prove that embeddings of compact groups are equalizers, and a number of results on pushouts (and more generally, amalgamated free products) in the category of compact groups. Call a family of compact-group embeddings $H\le G_i$ algebraically sound if the corresponding group-theoretic pushout embeds in its Bohr compactification. We (a) show that a family of normal embeddings is algebraically sound in the sense that $G_i$ admit embeddings $G_i\le G$ into a compact group which agree on $H$; (b) give equivalent characterizations of coherently embeddable families of normal embeddings in representation-theoretic terms, via Clifford theory; (c) characterize those compact connected Lie groups $H$ for which all finite families of normal embeddings $H\trianglelefteq G_i$ are coherently embeddable (not having central 2-tori is a sufficient, but not necessary condition), and (d) show that families of split embeddings of compact groups are always algebraically sound.
\end{abstract}


\vspace{.5cm}


\tableofcontents

\section*{Introduction}

The paper revolves around a number of category-theoretic issues pertaining to compact groups. To begin motivating the discussion in a slightly roundabout fashion (in part in order to highlight some of the thematic links between the various problems), recall \cite[Proposition 9]{reid-epi}, to the effect that the {\it epimorphisms} in the category of compact groups are precisely the surjective morphisms.

Here, epimorphisms are those of \cite[Definition 7.39]{ahs}: morphisms $c\xrightarrow{f} d$ in a category such that $\Hom(d,-)\xrightarrow{\circ f}\Hom(c,-)$ is one-to-one. This is the natural rendition of surjectivity in purely category-theoretic terms, and the two notions compete in {\it concrete} \cite[Definition 5.1]{ahs} categories (i.e. categories of sets and functions: \cat{CGp}, the category of compact groups, is an example).

An analogous result holds in the category of unital $C^*$-algebras (to give another example of an analytic/topological structure that is interesting category-theoretically), by either \cite[Proposition 2]{reid-epi} or \cite[Corollary 4]{hn-epi}. The latter source, in particular, deduces that result the stronger statement \cite[Theorem 6]{hn-epi} that embeddings of $C^*$-algebras are {\it equalizers} \cite[Definition 7.51]{ahs}. Circling back to compact groups, this is the first result below (\Cref{th:cpcteq}):

\begin{theoremN}
  In the category \cat{CGp} of compact groups, the equalizers are precisely the embeddings.  \qedhere
\end{theoremN}

In fact, an embedding $\iota:H\le G$ will be the equalizer of a {\it universal} pair of morphisms: the two structure maps $G\to G\coprod_HG$ into the {\it pushout} \cite[Definition 11.30]{ahs} of two copies of $\iota$. Such pushouts, and in fact colimits of diagrams in $\cat{CGp}$ consisting of families of embeddings $H\le G_i$ ({\it multi-}pushouts, in the sequel) are the main topic under consideration.

Such a multi-pushout $\coprod_{H,\cat{CGp}}G_i$ in the category \cat{CGp} is nothing but the {\it Bohr compactification} \cite[\S 2.10]{kan-comm} of the usual {\it free product of $G_i$ amalgamated over $H$} (\cite[Theorem 11.66 and preceding discussion]{rot-gp}, \cite[\S I.1.2]{ser_tr}, etc.).

This process of Bohr-compactifying the group-theoretic multi-pushout might, in principle, annihilate some of it; that this does in fact occasionally happen can be seen for instance from \Cref{ex:notemb}. On the other hand, a group-theoretic {\it coproduct} (i.e. multi-pushout of {\it trivial} embeddings) always embeds in its Bohr compactification: \cite[Proposition 1]{hul_map}. Similarly, finite families of {\it central} embeddings have the same property \cite[Theorem 8]{km_central-2}; not arbitrary central families though (\Cref{ex:heis}). 

We revisit the matter here. To introduce the language, call a family of \cat{CGp}-embeddings $\iota_i:H\le G_i$
\begin{itemize}
\item {\it algebraically sound} (\Cref{def:algsnd}) if the corresponding algebraic multi-pushout embeds in its Bohr compactification;
\item and {\it coherently embeddable} (\Cref{def:cohemb}) if there are embeddings $j_i:G_i\to G$ agreeing on $H$ (i.e. such that $j_i\iota_i$ are all equal).
\end{itemize}
The former property is formally stronger (\Cref{le:snd2coh}), and can be regarded as a global version of the latter. One of the results below is that implication's reversal for families of {\it normal} embeddings $H\trianglelefteq G_i$, along with other characterizations of such families (\Cref{th:normasound,th:whence}):

\begin{theoremN}
  For a family of normal \cat{CGp}-embeddings $\iota_i:H\trianglelefteq G_i$, the following conditions are equivalent:
  \begin{enumerate}[(a)]
  \item algebraically soundness;
  \item coherent embeddability;
  \item the existence of \cat{CGp}-morphisms $j_i:G_i\to G$ (not necessarily injective) with all $j_i\iota_i$ one-to-one and equal;
  \item The action of the group-theoretic coproduct $\Gamma:\coprod_{\cat{Gp}}G_i$ on $H$ factors through a compact subgroup of $\Aut(H)$;
  \item For every irreducible unitary $H$-representation $\rho$, its orbit under the conjugation $\Gamma$-action is finite and there are finite-dimensional unitary representations $\rho_i$ of $G_i$ such that $\rho_i|_H$ contains $\rho$ with multiplicity bounded in $i$.  \qedhere
  \end{enumerate}
\end{theoremN}

A natural offshoot of this discussion is to leverage it into a property of the amalgam $H$ itself: some $H$ have the property that every finite family of normal \cat{CGp}-embeddings $H\trianglelefteq G_i$ satisfies the equivalent conditions of the preceding theorem and some do not. For compact connected Lie groups the property in question relates in interesting ways to the compactness of the automorphism group under the uniform topology (\Cref{th:cpctconn}):

\begin{theoremN}
  Consider the following conditions for a compact connected Lie group $H$:
  \begin{enumerate}[(a)]
  \item\label{item:24} $H$ has no central 2-torus.
  \item\label{item:25} The Pontryagin dual $\widehat{Z(H)}$ of the center of $H$ has rank $\le 1$.  
  \item\label{item:26} The automorphism group $\Aut(H)$ is compact.
  \item\label{item:27} The automorphism group $\Aut(H)$ has a largest compact subgroup.
  \item\label{item:28} The poset of compact subgroups of $\Aut(H)$ is filtered.
  \item\label{item:29} The finite-order automorphisms of the identity component $Z_0(H)$ of the center $Z(H)$ induced by automorphisms of $H$ are precisely the identity and inversion.
  \item\label{item:30} Every finite family of normal \cat{CGp}-embeddings $\iota_i:H\trianglelefteq G_i$ is algebraically sound.  
  \end{enumerate}
  We then have
  \begin{equation*}
    \Cref{item:24}\iff\Cref{item:25}\iff\Cref{item:26}
    \qquad\Longrightarrow\qquad
    \Cref{item:27}\iff\Cref{item:28}\iff\Cref{item:29}\iff\Cref{item:30},
  \end{equation*}
  and the one-sided implication cannot, generally, be reversed.  \qedhere
\end{theoremN}

Finally, \Cref{subse:split} turns the normality of the embeddings around, so to speak: it is concerned with {\it split} embeddings $H\le G_i$, i.e. those of the form $H\le H_i\rtimes H$ for continuous actions $H\times H_i\to H_i$. The main result is that such families are always algebraically sound (\Cref{th:whenextcont,th:splitasnd}):

\begin{theoremN}
  A family of split \cat{CGp}-embeddings $H\le H_i\rtimes H$ is always algebraically sound.  \qedhere
\end{theoremN}

\subsection*{Acknowledgements}

This work was partially supported by NSF grant DMS-2001128. 

\section{Preliminaries}\label{se.prel}

We need some category-theoretic background and terminology, as covered, say, by \cite{ahs,mcl}: {\it (co)limits} \cite[\S 11]{ahs}, and in particular {\it (co)products} \cite[Definition 10.19 and \S 10.63]{ahs}, {\it pushouts} \cite[Definition 11.8]{ahs} and {\it pullbacks} \cite[Definition 11.30]{ahs}, etc.

Throughout the paper, we denote coproducts \cite[\S 10.63]{ahs} and/or pushouts \cite[Definition 11.30]{ahs} by `$\coprod$' symbols (rather than `$\ast$', as is sometimes customary). For instance, $\coprod c_i$ is the coproduct of the objects $c_i$ while the colimit of the diagram formed by morphisms $\varphi_i:c\to c_i$ will be denoted by $\coprod_{c}c_i$ (the morphisms being typically understood). We will also refer to colimits of the form $\coprod_c c_i$ as {\it multi-pushouts}, or just plain pushouts, when the context leaves no room for confusion.

Some symbols appearing frequently enough to set apart:
\begin{itemize}
\item $Z(H)$ is the center of a group $H$;
\item $H_0$ is the identity component of the topological group $H$;
\item and we write $Z_0(H)$ in place of $Z(H)_0$: the identity component of the center.
\end{itemize}

\section{(Multi-)Pushouts of Compact groups}\label{se:cpct}

The category of compact (always Hausdorff) groups will be denoted by \cat{CGp}, whereas that of ordinary groups is \cat{Gp}. The section gathers a number of category-theoretic results on \cat{CGp}, arguably in one circle of ideas, focusing mostly on pushouts:
\begin{itemize}
\item comparisons between the compact and purely algebraic versions of the construction;
\item the issue of whether or not the multi-pushout structure maps $G_i\to \coprod_{H,\cat{CGp}}G_i$ are embeddings;
\item whether embeddings $H_i\le G_i$ glue together to an embedding of $\cat{CGp}$-pushouts over a common $H\le H_i,G_i$, etc.
\end{itemize}

\subsection{Some generalities}\label{subse:gen}


We begin obliquely, by discussing {\it equalizers} \cite[Definition 7.51]{ahs}. In part, the motivation stems from $C^*$-algebra analogues. Recall from \cite[Theorem 6]{hn-epi} that embeddings of unital $C^*$-algebras are automatically pushouts; it follows from this, moreover, that an embedding
\begin{equation*}
  A \lhook\joinrel\xrightarrow{\quad\iota\quad} B
\end{equation*}
is the equalizer of the two embeddings $B\to B\coprod_AB$ into the self-pushout of $\iota$ (that the two maps are indeed embeddings is part of \cite[Theorem 4.2]{ped-psh}). The fact that epimorphisms in the category $\cC^*_1$ of unital $C^*$-algebras are surjective follows as an application \cite[Corollary 4]{hn-epi} (or rather as an application of a common precursor \cite[Theorem 3]{hn-epi}).

The surjectivity of $\cC^*_1$-epimorphisms is also proven differently in \cite[Proposition 2]{reid-epi}, which paper proves the analogue for compact groups (\cite[Proposition 9]{reid-epi}). It is natural, given this circle of ideas, to pose the analogue of the equalizer question for compact groups:

\begin{theorem}\label{th:cpcteq}
  An embedding
  \begin{equation*}
    H \lhook\joinrel\xrightarrow{\quad\iota\quad} G
  \end{equation*}
  of compact groups is an equalizer in the category \cat{CGp} of compact groups.
\end{theorem}
\begin{proof}
  A series of simplifications:
  \begin{enumerate}[(1)]
  \item\label{item:1} It is enough to find, for every $g\in G\setminus H$, a morphisms
    \begin{equation*}
      G\xrightarrow{\quad\varphi_{i,g}\quad}K_g,\ i=0,1
    \end{equation*}
    equal on $H$ but not at $g$, since we can afterwards glue them together into a pair of morphisms
    \begin{equation*}
      G\xrightarrow{\quad\prod_{g\not\in H}\varphi_{i,g}}\prod_{g\not\in H}K_g,\ i=0,1
    \end{equation*}
    equal on $H$ and nowhere else.

  \item\label{item:2} Fix $g\in G\setminus H$, in accordance with \Cref{item:1}. It will suffice to find a morphism $\varphi:G\to K$ and an element $k\in K$ which centralizes $\varphi(H)$ but not $\varphi(g)$, since we can then set
    \begin{equation*}
      \varphi_{0,g}:=\varphi\quad\text{and}\quad \varphi_{1,g}:=k\cdot \varphi\cdot k^{-1}. 
    \end{equation*}

  \item\label{item:3} In \Cref{item:2}, we can assume that $G$ is generated topologically by $H$ and $g$, and then find a $\varphi$ as required with $K=U(n)$ (the unitary $n\times n$ group, so that $\varphi$ is an $n$-dimensional representation).

    Indeed, once this is achieved for the inclusion $H\le G_1:=\overline{\langle H,x\rangle}$ into the closed subgroup of $G$ generated by $H$ and $x$ with a representation $\rho:G_1\to U(m)$, we can take $\varphi:G\to U(n)$ to be a finite-dimensional representation whose restriction back to $G_1$ contains $\rho$ as a summand (e.g. by {\it Frobenius reciprocity} \cite[(8.9)]{rob}, $\varphi$ can be chosen as a summand of the {\it induced representation} \cite[preceding (8.7)]{rob} $\mathrm{Ind}_{G_1}^G\rho$). 
    
  \item\label{item:4} Part \Cref{item:3} reduces the problem to finding, for a proper inclusion $H\le G$, a finite-dimensional $G$-representation $\varphi:G\to U(n)$ whose group of unitary intertwiners (i.e. the centralizer $C_{U(n)}(\varphi(G))$ of $\varphi(G)$ in $U(n)$) is strictly contained in $C_{U(n)}(\varphi(H))$.
    
    Because $H\subset G$ is proper, {\it Tannaka-Krein duality} \cite[Theorem 20]{js} implies that either
    \begin{itemize}
    \item some irreducible $G$-representation $\varphi$ fails to remain irreducible when restricted to $H$;
    \item or mutually non-isomorphic irreducible representations $\varphi_i$, $i=0,1$ of $G$ become isomorphic when restricted to $H$. 
    \end{itemize}
    In the first case $\varphi$ will do for our purposes, while in the second we can set $\varphi:=\varphi_0\oplus \varphi_1$.  
  \end{enumerate}
  This concludes the proof. 
\end{proof}

In particular, as for $C^*$-algebras, an embedding $H\le G$ of compact groups is the equalizer of the two structure maps $G\to G\coprod_HG$ of its own self-pushout. It is perhaps worth noting at this stage that some care is needed in handling compact-group pushouts: as \Cref{ex:notemb} shows, the corresponding plain, group-theoretic pushout need not embed in its compact-group counterpart.

\begin{example}\label{ex:notemb}
  We will consider the pushout $G_0\coprod_{\bT^2}G_1$ of two compact groups $G_i$ over a common copy of the 2-torus $\bT^2$. The two are defined as follows:
  \begin{itemize}
  \item $G_0$ is the semidirect product $\bT^2\rtimes (\bZ/4)$, where a generator of $\bZ/4$ acts on the {\it Pontryagin dual} \cite[\S 3.5]{de} $\bZ^2\cong \widehat{\bT^2}$ via the matrix $\begin{pmatrix}0&-1\\1&\phantom{-}0\end{pmatrix}$;
  \item and similarly, $G_1$ is the semidirect product $\bT^2\rtimes (\bZ/6)$ with a generator of the second factor acting on $\bZ^2$ by $\begin{pmatrix}0&-1\\1&\phantom{-}1\end{pmatrix}$.
  \end{itemize}
  The pushout turns out to be very small: according to \Cref{le:smallpush}, it collapses the amalgam $\bT^2$ completely. It follows that $G_i$, $i=1,2$ do not admit embeddings into a common compact group $G$ which agree on $\bT^2$. In particular, the group-theoretic pushout $G_0\coprod_{\bT^2}G_1$, equipped with the finest group topology making the embeddings of the $G_i$ continuous \cite[Theorem 2.2 and its proof]{mor_free}, does not embed into its {\it Bohr compactification} \cite[\S 2.10]{kan-comm}.
\end{example}

\begin{remark}\label{re:bohr}
  A word of caution regarding the term `Bohr compactification': here, it simply means, for any topological group $G$, the universal compact group receiving a continuous morphism from $G$; this usage is compatible, say, with \cite[\S 1]{hlm_bohr}. In other contexts (e.g. \cite[preceding Proposition 1.3]{hul_map}; a paper also relevant to the subsequent discussion) the phrase has a somewhat more restrictive meaning: the morphism from $G$ is required to be a continuous embedding.

  In the more restrictive sense, such a gadget does not always exist ($G$ might not embed into a compact group at all: \Cref{ex:notemb}). On the other hand, in the broader sense operative in this paper the Bohr compactification always exists; one route to the construction, for instance, is via one of the {\it adjoint functor theorems} of, say, \cite[\S 18]{ahs}.
\end{remark}

\begin{lemma}\label{le:smallpush}
  Set $G_0=\bT^2\rtimes \bZ/4$ and $G_1=\bT^2\rtimes \bZ/6$, as in \Cref{ex:notemb}. The $\cat{CGp}$-pushout $G_0\coprod_{\bT^2}G_1$ discussed there is the Bohr compactification of the coproduct $(\bZ/4)*(\bZ/6)$.
\end{lemma}
\begin{proof}
  The claim amounts to showing that any pair of continuous morphisms $\varphi_i:G_i\to G$ which agree on $\bT^2$ must annihilate the latter.

  To see this, set
  \begin{equation*}
    K:=\bT^2/\ker~\varphi_i|_{\bT^2}.
  \end{equation*}
  Every character
  \begin{equation*}
    \chi\in\widehat{K}\subseteq \widehat{\bT^2}\cong \bZ^2
  \end{equation*}
  is contained in the restriction of some finite-dimensional $G$-representation $\rho$. The restriction $\rho|_{H}$ contains, along with $\chi$, all of its $\bZ/4$- and $\bZ/6$-iterates. Since the two groups generate $\mathrm{SL}(2,\bZ)$ \cite[equation (3) following Proposition 1.18]{Hain11}, all non-zero $\chi\in \bZ^2$ have infinite orbit. The only way $\rho$ could be finite-dimensional, then, is for $\chi$ to be trivial. Because it was an arbitrary character of the compact abelian group $K$, the latter must be trivial as well.
\end{proof}

\begin{remarks}\label{res:gpnotemb}
  \begin{enumerate}[(1)]

  \item\label{item:5} Note, incidentally, that the same \Cref{ex:notemb} also shows that in general, in the category of compact groups, $G_i\to G_0\coprod_H G_1$ need not be embeddings. Contrast this with plain-group pushouts, where the respective maps {\it are} embeddings \cite[Theorem 11.67 (i)]{rot-gp}, or again with the $C^*$ version of this same positive result in \cite[Theorem 4.2]{ped-psh}.
  \item\label{item:7} On the other hand, in any category (where the relevant pushouts exist), the canonical maps $G\to G\coprod_HG$ into the {\it self}-pushout of a single map $\iota:H\to G$ are always left-invertible, with the common left inverse given by the morphism
    \begin{equation}\label{eq:ggg}
      G\coprod_HG\to G\text{ with components }G\xrightarrow{\quad\id\quad} G.
    \end{equation}
    In particular, such maps $G\to G\coprod_HG$ are embeddings in categories that are {\it concrete} (over \cat{Set}; e.g. the category \cat{CGp}), i.e. those whose objects are (essentially) sets and whose morphisms are functions between those sets \cite[Definition 5.1]{ahs}.
  \item\label{item:6} In fact, the pathology in \Cref{ex:notemb} all stemmed from one's inability to embed both $G_i$ (separately) into a single compact group so as to have the embeddings agree on the common subgroup $H=\bT^2\le G_i$. \Cref{le:ifcommonemb} formalizes this, given that the {\it monomorphisms} \cite[Definition 7.32]{ahs} in \cat{CGp} are precisely the one-to-one morphisms (as can easily be seen).
  \end{enumerate}
\end{remarks}

\begin{lemma}\label{le:ifcommonemb}
  Let $\iota_i : H\to G_i$, $i=0,1$ be monomorphisms in a category with pushouts.

  If $G_i$ admit monomorphisms $j_i$ into an object $G$ such that $j_i\iota_i:H\to G$ coincide then the structure maps $G_i\to G_0\coprod_HG_1$ are monic.
\end{lemma}
\begin{proof}
  . The conclusion now follows from the fact that the top composition in
  \begin{equation*}
    \begin{tikzpicture}[auto,baseline=(current  bounding  box.center)]
      \path[anchor=base] 
      (0,0) node (l) {$G_i$}
      +(2,.5) node (u) {$G$}
      +(2,-.5) node (d) {$G_0\coprod_HG_1$}
      +(4,0) node (r) {$G\coprod_HG$}
      ;
      \draw[->] (l) to[bend left=6] node[pos=.5,auto] {$\scriptstyle $} (u);
      \draw[->] (u) to[bend left=6] node[pos=.5,auto] {$\scriptstyle $} (r);
      \draw[->] (l) to[bend right=6] node[pos=.5,auto,swap] {$\scriptstyle $} (d);
      \draw[->] (d) to[bend right=6] node[pos=.5,auto,swap] {$\scriptstyle $} (r);
    \end{tikzpicture}
  \end{equation*}
  is monic (the left-hand arrow is by assumption, and the right-hand arrow is left-invertible, as in \Cref{res:gpnotemb} \Cref{item:7}). The bottom composition is thus also monic, and hence so is the bottom left-hand arrow, because monomorphisms are stable under left cancellation: if $fg$ is monic, so is $g$.
\end{proof}

A handy phrase for the sort of setup discussed in \Cref{le:ifcommonemb}:

\begin{definition}\label{def:cohemb}
  For monomorphisms $\iota_i:H\to G_i$ in a category with pushouts, the objects $G_i$ are {\it ($\iota_i$)- or $H$-coherently embeddable} (or just plain {\it coherently embeddable} when the context is clear) if there are monomorphisms $j_i:G_i\to G$ with equal $j_i\iota_i:H\to G$.

  Alternatively, we might refer to the pair $(\iota_i)$ itself as being coherently embeddable.
\end{definition}

\Cref{pr:intsmall} says, roughly, that under conditions that obtain in the applications below ``the intersection of $G_i$ in a pushout $G_1\coprod_H G_1$ equals $H$''. Intersections, which do not make sense as such in arbitrary categories, can be cast instead as pullbacks; and to make sense of the statement, recall that a monomorphism is {\it regular} \cite[Definition 7.56]{ahs} if it is an equalizer.

\begin{proposition}\label{pr:intsmall}
  Consider a coherently-embeddable pair $\iota_i:H\to G_i$, $i=0,1$ of monomorphisms in a category with pushouts and in which monomorphisms are regular.

  The diagram
  \begin{equation*}
    \begin{tikzpicture}[auto,baseline=(current  bounding  box.center)]
      \path[anchor=base] 
      (0,0) node (l) {$H$}
      +(2,.5) node (u) {$G_0$}
      +(2,-.5) node (d) {$G_1$}
      +(4,0) node (r) {$G_0\coprod_H G_1$}
      ;
      \draw[->] (l) to[bend left=6] node[pos=.5,auto] {$\scriptstyle \iota_0$} (u);
      \draw[->] (u) to[bend left=6] node[pos=.5,auto] {$\scriptstyle $} (r);
      \draw[->] (l) to[bend right=6] node[pos=.5,auto,swap] {$\scriptstyle \iota_1$} (d);
      \draw[->] (d) to[bend right=6] node[pos=.5,auto,swap] {$\scriptstyle $} (r);
    \end{tikzpicture}
  \end{equation*}
  is then a pullback.
\end{proposition}
\begin{proof}
  Consider embeddings $j_i:G_i\to G$ as in \Cref{def:cohemb}, which will produce a pushout $G\coprod_HG$ fitting into the diagram
  \begin{equation*}
    \begin{tikzpicture}[auto,>=stealth,baseline=(current  bounding  box.center)]
      \path[anchor=base] 
      (0,0) node (l) {$H$}
      +(2,.5) node (u) {$G_0$}
      +(2,-.5) node (d) {$G_1$}
      +(4,0) node (r) {$G_0\coprod_H G_1$}
      +(4,2) node (uu) {$G$}
      +(4,-2) node (dd) {$G$}
      +(7,0) node (rr) {$G\coprod_H G$}
      ;
      \draw[->] (l) to[bend left=6] node[pos=.5,auto] {$\scriptstyle \iota_0$} (u);
      \draw[->] (u) to[bend left=6] node[pos=.5,auto] {$\scriptstyle $} (r);
      \draw[->] (l) to[bend right=6] node[pos=.5,auto,swap] {$\scriptstyle \iota_1$} (d);
      \draw[->] (d) to[bend right=6] node[pos=.5,auto,swap] {$\scriptstyle $} (r);
      \draw[->] (r) to[bend right=0] node[pos=.5,auto] {$\scriptstyle j_0*j_1$} (rr);
      \draw[->] (l) to[bend right=-25] node[pos=.5,auto] {$\scriptstyle $} (uu);
      \draw[->] (l) to[bend right=25] node[pos=.5,auto] {$\scriptstyle $} (dd);
      \draw[->] (uu) to[bend right=-25] node[pos=.5,auto] {$\scriptstyle $} (rr);
      \draw[->] (dd) to[bend right=25] node[pos=.5,auto] {$\scriptstyle $} (rr);
      \draw[->] (u) to[bend right=-10] node[pos=.5,auto,swap] {$\scriptstyle j_0$} (uu);
      \draw[->] (d) to[bend right=10] node[pos=.5,auto] {$\scriptstyle j_1$} (dd);
    \end{tikzpicture}
  \end{equation*}
  The outer square is a pushout and the outer left-hand maps are monic and hence regular monic by assumption. The dual result to \cite[Proposition 11.33]{ahs} then implies that that outer square is a {\it pulation square}, i.e. a pullback as well as a pushout. In short, the statement of the proposition holds for pairs of {\it identical} monomorphisms $H\to G$.

  The rest is a straightforward diagram chase: if $f_i:X\to G_i$ compose with $G_i\to G_0\coprod_H G_1$ to give equal morphisms into the latter object, then they further give equal morphisms into $G\coprod_H G$ and hence
  \begin{equation*}
    \begin{tikzpicture}[auto,baseline=(current  bounding  box.center)]
      \path[anchor=base] 
      (0,0) node (l) {$X$}
      +(2,.5) node (u) {$H$}
      +(4,0) node (r) {$G$}
      ;
      \draw[->] (l) to[bend left=6] node[pos=.5,auto] {$\scriptstyle f$} (u);
      \draw[->] (u) to[bend left=6] node[pos=.5,auto] {$\scriptstyle $} (r);
      \draw[->] (l) to[bend right=6] node[pos=.5,auto,swap] {$\scriptstyle j_i f_i$} (r);
    \end{tikzpicture}
  \end{equation*}
  commute for a unique map $f$. But then 
  \begin{equation*}
    \begin{tikzpicture}[auto,baseline=(current  bounding  box.center)]
      \path[anchor=base] 
      (0,0) node (l) {$X$}
      +(2,.5) node (u) {$H$}
      +(4,0) node (r) {$G_i$}
      ;
      \draw[->] (l) to[bend left=6] node[pos=.5,auto] {$\scriptstyle f$} (u);
      \draw[->] (u) to[bend left=6] node[pos=.5,auto] {$\scriptstyle $} (r);
      \draw[->] (l) to[bend right=6] node[pos=.5,auto,swap] {$\scriptstyle f_i$} (r);
    \end{tikzpicture}
  \end{equation*}
  must also commute, $j_i$ being embeddings.
\end{proof}

\begin{corollary}\label{cor:cgp-intsmall}
  Let $\iota_i:H\le G_i$, $i=0,1$ be a coherently-embeddable pair of monomorphisms in \cat{CGp}.

  In the $\cat{CGp}$-pushout $G_0\coprod_H G_1$ we have $G_0\cap G_1 = H$.
\end{corollary}
\begin{proof}
  This is a direct application of \Cref{pr:intsmall}, given that
  \begin{itemize}
  \item $G_i\to G_0\coprod_H G_1$ are indeed embeddings (\Cref{le:ifcommonemb});
  \item the intersection in the statement is precisely the pullback of those embeddings;
  \item and in \cat{CGp} embeddings are equalizers (i.e. monomorphisms are regular) by \Cref{th:cpcteq}. 
  \end{itemize}
\end{proof}

\begin{remark}\label{re:plaingpint}
  The ordinary-group version of \Cref{cor:cgp-intsmall} holds without qualification, i.e. for arbitrary pushouts (of embeddings) in \cat{Gp} \cite[Theorem 11.67 (ii)]{rot-gp}.
\end{remark}


Consider, now, a compact-group (multi-)pushout $\coprod_H G_i$. The map
\begin{equation}\label{eq:alg2cpct}
  \coprod_{H,\cat{Gp}} G_i\to \coprod_H G_i
\end{equation}
from the {\it ordinary}-group pushout sometimes is an embedding, and sometimes isn't (\Cref{ex:notemb}); the following term captures the relevant notion:

\begin{definition}\label{def:algsnd}
  A collection $\iota_i:H\to G_i$ of compact-group embeddings is {\it algebraically sound} if the map \Cref{eq:alg2cpct} from the corresponding group-theoretic pushout to its compact-group counterpart is one-to-one.
\end{definition}

\begin{remarks}\label{res:hulan}
  \begin{enumerate}[(1)]
  \item In these terms, \cite[Proposition 1]{hul_map} can be phrased as saying that every family $\{1\}\to G_i$ of trivial embeddings in \cat{CGp} is algebraically sound.
    
  \item\label{item:9} Motivated in part by the previous observation and in part by the analogy to algebras, where families of embeddings $B\to A_i$ that {\it split} as bimodule maps (i.e. admit left inverses in the category of $B$-bimodules) are particularly well-behaved for the purpose of describing their associated pushouts \cite[Corollary 8.1]{bg-diamond}, one might hope that families of split embeddings in \cat{CGp} are similarly pleasant. We address the matter below.


  \end{enumerate}  
\end{remarks}

One immediate remark:

\begin{lemma}\label{le:snd2coh}
  An algebraically sound family of embeddings of compact groups is coherently embeddable in \cat{CGp} in the sense of \Cref{def:cohemb}. 
\end{lemma}
\begin{proof}
  For a family of embeddings $H\le G_i$, algebraic soundness ensures in particular that the structure maps $G_i\to \coprod_{H,\cat{CGp}}G_i$ of the multi-pushout are embeddings. They agree on $H$, so said maps will witness the coherent embeddability.
\end{proof}

\Cref{le:famsubfam} records another observation, useful on occasion. Recall that a {\it filtered} (or {\it directed} \cite[Definition 11.1]{wil-top}) poset $(P,\le)$ is one for which any finite set of elements is dominated by another.

\begin{lemma}\label{le:famsubfam}
  Let $\iota_i:H\le G_i$, $i\in I$ be a family of embeddings in \cat{CGp}.
  \begin{enumerate}[(1)]
  \item\label{item:fam-subfam} If $(\iota_i)$ is algebraically sound, so is every subfamily.

  \item\label{item:subfam-fam} Conversely, if $(\iota_i)$ is filtered union of algebraically sound subfamilies $(\iota_j)_{j\in J}$ for which
    \begin{equation*}
      \coprod_{H,\cat{Gp}}G_j
      \to
      \coprod_{H,\cat{Gp}}G_i
    \end{equation*}
    splits, then $(\iota_i)$ is algebraically sound. 
  \end{enumerate}
\end{lemma}
\begin{proof}
  Part \Cref{item:fam-subfam} is immediate: for a subfamily $\iota_j:H\to G_j$, $j\in J\subseteq I$ we have a factorization
  \begin{equation*}
    \begin{tikzpicture}[auto,baseline=(current  bounding  box.center)]
      \path[anchor=base] 
      (0,0) node (l) {$\coprod_{\cat{Gp}}G_j$}
      +(4,.5) node (u) {$\coprod_{\cat{Gp}}G_i$}
      +(4,-.5) node (d) {$\coprod_{\cat{CGp}}G_j$}
      +(8,0) node (r) {$\coprod_{\cat{CGp}}G_i$}
      ;
      \draw[->] (l) to[bend left=6] node[pos=.5,auto] {$\scriptstyle $} (u);
      \draw[->] (u) to[bend left=6] node[pos=.5,auto] {$\scriptstyle $} (r);
      \draw[->] (l) to[bend right=6] node[pos=.5,auto,swap] {$\scriptstyle $} (d);
      \draw[->] (d) to[bend right=6] node[pos=.5,auto,swap] {$\scriptstyle $} (r);
    \end{tikzpicture}
  \end{equation*}
  with the upper left-hand map one-to-one. If the upper {\it right}-hand map is also injective (the hypothesis), then so is the bottom left-hand morphism.

  For part \Cref{item:subfam-fam}, note that the algebraic multi-pushout $\coprod_{H,\cat{Gp}}G_i$ is the union of the pushouts induced by the finite subfamilies $F\subseteq I$. It is thus enough to argue that an element $x$ of, say $\coprod_{H,\cat{Gp}}G_f$ (with $f\in F$) is not annihilated by the mapping to $\coprod_{H,\cat{CGp}}G_i$. Because $I$ is assumed to be a {\it filtered} union of the $J$, we have $F\subset J$ for some such $J$.

  We are assuming that $x$ is not trivial in $\coprod_{H,\cat{CGp}}G_j$, and
  \begin{equation*}
    \coprod_{H,\cat{CGp}}G_j\to \coprod_{H,\cat{CGp}}G_i
  \end{equation*}
  splits because its algebraic counterpart does. It must thus be injective, and $x$ is non-trivial in $\coprod_{H,\cat{CGp}}G_i$.
\end{proof}

\begin{remark}\label{re:km}
  The most natural application of \Cref{le:famsubfam} \Cref{item:subfam-fam} is to leverage algebraic embeddability from {\it finite} subfamilies to the entire family.
  
  One might wonder whether that portion of the lemma goes through without additional assumptions on $(\iota_i)$, such the splittings. That {\it some} condition is needed can be seen even from examples of central embeddings. {\it Finite} families of central embeddings in \cat{CGp} are algebraically sound by \cite[Theorem 8]{km_central-2} (and induction, since that result is about pairs of central embeddings), but infinite families need not be: see \Cref{ex:heis}.
\end{remark}

\begin{example}\label{ex:heis}
  For positive integers $i$, let $G_i$ be the {\it Heisenberg group $H_{2^i}$ over $\bZ/2^i$}, of order $(2^i)^3=2^{3i}$ \cite[\S 2]{schulte}: the multiplicative group of unipotent $3\times 3$ matrices with entries in $\bZ/2^i$.

  Each $G_i$ has a center isomorphic to (the additive group) $\bZ/2^i$, consisting of matrices
  \begin{equation*}
    \begin{pmatrix}
      1&0&x\\
      0&1&0\\
      0&0&1
    \end{pmatrix},\ x\in \bZ/2^i,
  \end{equation*}
  and $H\le G_i$ will be the embeddings of a single copy of $\bZ/2$ into these cyclic centers. I claim that this family is not even coherently embeddable: every family of $\cat{CGp}$-morphisms $\varphi:G_i\to K$ equal on $H$ must in fact annihilate $H$.

  Indeed, were it not so, the non-trivial character of $H$ would be the restriction along (any of) $\varphi_i$ of a finite-dimensional representation $\rho$ of $K$. But the classification of the representations of $H_{2^i}$ \cite[Theorem 3]{schulte} makes it clear that the only irreducible representations that restrict to a non-trivial character on $\bZ/2$ are $2^i$-dimensional. There is thus no bound on the requisite dimension of $\rho|_{H_{2^i}}$, and we have a contradiction.
\end{example}


\subsection{Families of normal embeddings}\label{subse:norm}

Given the various ways algebraic soundness can fail for families of seemingly well-behaved embeddings (finite families of normal embeddings in \Cref{ex:notemb}, arbitrary families of central embeddings in \Cref{ex:heis}), it seems pertinent to try to isolate (necessary and) sufficient conditions on a family of {\it normal} embeddings that will ensure its algebraic soundness. The following result does this.

\begin{theorem}\label{th:normasound}
  A family $\iota:H\trianglelefteq G_i$ of normal embeddings of compact groups is algebraically sound if and only if it is coherently embeddable in \cat{CGp} in the sense of \Cref{def:cohemb}.
\end{theorem}
\begin{proof}
  The forward implication is the easy one (\Cref{le:snd2coh}), so we focus on the converse: for normal embeddings, coherent embeddability entails algebraic soundness. We will simplify the setup progressively.

  \begin{enumerate}[(I)]

  \item {\bf Reducing the problem to families of identical maps.} By hypothesis, there are \cat{CGp}-embeddings $j_i:G_i\to K$ agreeing on $H\le G_i$; we will identify $H$ with its image through the common map $j_i\iota_i$. Because every $j_i G_i\le K$ normalizes $H\le K$, we may as well replace $K$ with that normalizer, and assume throughout that $H$ is normal in $K$. Further, because the purely algebraic map
    \begin{equation*}
      \coprod_{H,\cat{Gp}}G_i\to \coprod_{H,I,\cat{Gp}}K
    \end{equation*}
    is an embedding (the codomain denoting the $I$-fold self-multi-pushout of $H\le K$), it will be enough to argue that the family consisting of $|I|$ copies of $H\le K$ is algebraically sound.

  \item {\bf Reduction to finite families of identical maps.} Note that for a family $(\iota_i)$ of identical maps $\iota_i=\iota:H\to K$, the splitting hypothesis of \Cref{le:famsubfam} \Cref{item:subfam-fam} is satisfied: given a subfamily indexed by $J\subseteq I$, we have a splitting of
    \begin{equation*}
      \coprod_{H,J,\cat{Gp}}K \to \coprod_{H,I,\cat{Gp}}K
    \end{equation*}
    defined as the identity on the $J$-indexed copies of $K$ and
    \begin{equation*}
      \text{$i$-indexed }K\xmapsto{\quad\id\quad}\text{some $j$-indexed }K\text{ for $i\not\in J$}. 
    \end{equation*}
    It follows from \Cref{le:famsubfam} \Cref{item:subfam-fam} that it suffices to prove the algebraic soundness of a finite family.

  \item{\bf Reduction to pairs of identical maps.} Per the preceding point, consider a normal embedding $\iota:H\trianglelefteq K$ and a positive integer $n$. By \Cref{le:famsubfam} \Cref{item:fam-subfam}, the family $(\iota)_{[n]}$ consisting of $n$ copies of $\iota$ will be algebraically sound provided a family consisting of {\it more} copies of $\iota$ is.

    We can thus replace $n$ with a power of 2 dominating it, and then leverage the result for just two maps by induction: the algebraic soundness of $(\iota)_{[2]}$ entails that of $(\iota)_{[2^2]}$, then $(\iota)_{[2^3]}$, etc.

  \item {\bf Reduction to a normal embedding and a central split one.} The claim has now boiled down to the injectivity of
    \begin{equation}\label{eq:khkh}
      K\coprod_{H,\cat{Gp}}K\to K\coprod_{H,\cat{CGp}}K
    \end{equation}
    for a normal embedding $H\trianglelefteq K$ in \cat{CGp}. There are compatible actions of $\bZ/2$ on both groups, interchanging the two copies of $K$ and acting trivially on $H$. In both cases we have isomorphisms
    \begin{equation*}
      (K\coprod_H K)\rtimes \bZ/2\cong K\coprod_H (H\times \bZ/2)
    \end{equation*}
    (regardless of category): identify the second copy of $K$ in the semidirect product on the left with the conjugate $\sigma K\sigma$ of $K$ on the right-hand side, where $1\ne \sigma\in \bZ/2$. Plainly, \Cref{eq:khkh} is injective if and only if
    \begin{equation*}
      \left(K\coprod_{H,\cat{Gp}}K\right)\rtimes \bZ/2\to \left(K\coprod_{H,\cat{CGp}}K\right) \rtimes \bZ/2
    \end{equation*}
    is, so the issue now is to prove the algebraic soundness of the pair
    \begin{equation}\label{eq:hkz2}
      H\trianglelefteq K\quad\text{and}\quad H\le H\times \bZ/2. 
    \end{equation}

  \item {\bf Conclusion.} To finish, we observe that the proof of \cite[Theorem 8]{km_central-2} applies to also deliver the algebraic soundness of the pair \Cref{eq:hkz2}:
    \begin{itemize}
    \item Elements of $K\coprod_H (H\times \bZ/2)$ belonging to the kernel of
      \begin{equation}\label{eq:2ktimesz2}
        K\coprod_H (H\times \bZ/2)\to K\times \bZ/2
      \end{equation}
      are mapped non-trivially into $(K/H)\coprod_{H,\cat{Gp}}(\bZ/2)$, which embeds into {\it its} Bohr compactification by \cite[Proposition 1]{hul_map}.

    \item On the other hand, elements {\it not} annihilated by \Cref{eq:2ktimesz2} are not an issue, since $K\times \bZ/2$ itself is compact. 
    \end{itemize}

  \end{enumerate}
  This concludes the proof. 
\end{proof}

\begin{remark}\label{re:locgl}
  The interesting implication `$\Leftarrow$' of \Cref{th:normasound} can be thought of as bootstrapping a(n at least partially) ``local'' embeddability condition (each $G_I$, individually, is embeddable into a common $G$ via embeddings agreeing on $H$) into a ``global'' one (the multi-pushout as a whole embeds into a compact group).

  For a similarly-flavored statement in a somewhat different context, consider \cite[Theorem 2]{li-shen_rfd}: if two unital $C^*$-algebras $A$ and $B$ admit embeddings into some product of matrix algebras which agree on a common finite-dimensional $D\le A,B$, then the entire pushout $A\coprod_D B$ in the category of unital $C^*$-algebras embeds into a product of matrix algebras.
\end{remark}

While \Cref{th:normasound} does provide some purchase, it is somewhat dissatisfying: one would expect, perhaps, a more explicit framing of {\it when} coherent embeddability obtains. \Cref{th:whence} addresses this, after some preliminaries.

To make sense of the subsequent material, we recall very briefly some of the basics of {\it Clifford theory} (originally developed in \cite{clif} for finite groups), as applicable to normal embeddings $H\trianglelefteq G$ of compact groups. There is also a review of the finite-group theory in \cite[\S 2]{cst}, and a summary of a quantum-group (i.e. Hopf-algebra) version, also covering plain compact groups, in \cite[Theorem 0.2]{chi_relcent}.

We denote by $\widehat{\bullet}$ the construction attaching to a compact group its set of (isomorphism classes of) irreducible unitary representations. The normal $\cat{CGp}$-embedding $H\trianglelefteq G$ induces an equivalence relation `$\sim$' on $\widehat{H}$ via
\begin{equation*}
  \rho\sim\rho'\Longleftrightarrow \exists \pi\in \widehat{G},\quad \rho,\rho'\le \pi|_H
\end{equation*}
(`$\le$' meaning `is a subrepresentation (hence a summand) of'). It turns out, moreover, that for every irreducible $\pi\in \widehat{G}$, the restriction $\pi|_H$ decomposes as a sum of copies of $\rho$ precisely ranging over an equivalence class (and specifically, a conjugacy class under the action of $G$ on $H$), each appearing with the same multiplicity depending on $\pi$:
\begin{equation}\label{eq:pidec}
  \forall \pi\in \widehat{G},\ \pi|_H\cong \bigoplus_{\rho\in C_{\pi}}\rho^{\oplus m_{\pi}}\text{ for a class }C_{\pi}\text{ of `$\sim$' and some }m_{\pi}\in \bZ_{>0}.
\end{equation}

Coherent embeddability is a kind of boundedness assumption on all of this representation-theoretic data, ensuring that the two qualitative phenomena exhibited (respectively) in \Cref{ex:notemb,ex:heis} do not occur:
\begin{itemize}
\item the sizes of the equivalence classes $C_{\pi}$
\item and the multiplicities $m_{\pi}$
\end{itemize}
all stay bounded as $\pi$ range over representations of the various $G_i$ dominating a given element of $\widehat{H}$. Some terminology will help formalize this intuition.

\begin{definition}\label{def:classmult}
  \begin{enumerate}[(1)]
  \item Let $\iota: H\trianglelefteq G$ be a normal embedding of compact groups and $\rho\in\widehat{H}$.
    \begin{itemize}
    \item The {\it Clifford relation} $\sim_{\iota}$ on $\widehat{H}$ is the equivalence relation `$\sim$' of the discussion above.
    \item The {\it Clifford class} $C_{\rho}=C_{\rho|\iota}$ of $\rho$ (relative to the embedding $\iota$) consists of all of its $G$-conjugates. Equivalently, it is the class of $\rho$ under the Clifford relation $\sim_{\iota}$.
    \item The {\it Clifford $\pi$-multiplicity} $m_{\rho|\pi} = m_{\rho|\pi,\iota}$ of $\rho\in\widehat{H}$ with respect to $\pi\in\widehat{G}$ is the $m_{\pi}$ of \Cref{eq:pidec}, understood to be 0 if $\pi|_H$ does not contain $\rho$.
    \item The {\it Clifford multiplicity} $m_{\rho}=m_{\rho|\iota}$ is the smallest of all (strictly) positive $m_{\rho|\pi}$, for $\pi\in\widehat{G}$.
    \end{itemize}

  \item Now consider a family $\iota_i:H\trianglelefteq G_i$, $i\in I$ of normal embeddings and $\rho\in\widehat{H}$.
    \begin{itemize}
    \item The {\it Clifford relation} $\sim_{(\iota_i)}$ is the finest equivalence relation on $\widehat{H}$ coarser than all $\sim_{\iota_i}$.

      Because each individual relation $\sim_{\iota_i}$ is nothing but the orbit relation for the conjugacy action of $G_{\iota_i}$ on $\widehat{H}$, the coarser $\sim_{(\iota_i)}$ is similarly the orbit relation of the conjugacy action of $\coprod_{\cat{Gp}}G_i$ on $\widehat{G}$. This makes sense, regardless of any topological issues: that algebraic coproduct certainly acts on $H$ by automorphisms in the category $\cat{CGp}$, and hence acts on $\widehat{H}$.
    \item The {\it Clifford class} $C_{\rho}=C_{\rho|(\iota_i)}$ is the class of $\rho$ with respect to $\sim_{(\iota_i)}$.

      Or: the conjugacy class of $\rho$ under the action of $\coprod_{\cat{Gp}}G_i$,
    \item The {\it Clifford multiplicity} $m_{\rho}=m_{\rho|(\iota_i)}$ (positive integer or $\infty$) is the supremum of all $m_{\rho|\iota_i}$ for varying $i\in I$.
    \end{itemize}
  \end{enumerate}
\end{definition}

\begin{remark}
  As alluded to above, \Cref{ex:notemb} is built so that the Clifford classes of the pair of embeddings are infinite (or rather, most of them are). On the other hand, in \Cref{ex:heis} the Clifford classes are singletons because $H$ is central, but (some of) the Clifford multiplicities are infinite.
\end{remark}

\begin{theorem}\label{th:whence}
  For a family $\iota_i:H\trianglelefteq G_i$ of normal compact-group embeddings, the two conditions of \Cref{th:normasound} are also equivalent to the following:
  \begin{enumerate}[(a)]
  \item\label{item:wce} There are \cat{CGp}-morphisms $f_i:G_i\to G$ with all $f_i\iota_i$ equal and injective.
  \item\label{item:fin} The requirement $\cat{FIN}_{(\iota_i)}$:
    \begin{equation}\label{eq:fincond}
      \text{
        For every $\rho\in \widehat{H}$ the Clifford class $C_{\rho|(\iota_i)}$ and Clifford multiplicity $m_{\rho|(\iota_i)}$ are both finite.
      }
    \end{equation}
  \end{enumerate}
\end{theorem}

Note that \Cref{item:wce} is a formal strengthening of the local criterion of coherent embeddability: the $f_i$ themselves are no longer required to be injective. We refer to this condition as {\it weak} coherent embeddability, and first isolate the more straightforward implication.

\begin{lemma}\label{le:ce2fin}
  If a family $\iota_i:H\trianglelefteq G_i$ of normal compact-group embeddings is weakly coherently embeddable, then condition $\cat{FIN}_{(\iota_i)}$ of \Cref{th:whence} holds. 
\end{lemma}
\begin{proof}
  Fix maps $f_i:G_i\le G$ as in \Cref{th:whence} \Cref{item:wce}, assuming furthermore that the resulting embedding $f_i\iota_i:H\le G$ is normal (as in the proof of \Cref{th:normasound}: substitute the normalizer of $H$ in $G$ for the latter).

  
  To conclude, simply note that the conjugacy action of $\coprod_{\cat{Gp}}G_i$ on $H$ factors through that of $G$, so that 
  \begin{equation*}
    C_{\rho|(\iota_i)}\subseteq C_{\rho|j_i\iota_i} = \text{finite set of $G$-conjugates of $\rho$},
  \end{equation*}
  and
  \begin{equation*}
    m_{\rho|(\iota_i)}\le m_{\rho|(j_i\iota_i)}<\infty.
  \end{equation*}
  The finiteness of the values on the right, then, entails that of the left-hand items.
\end{proof}

\pf{th:whence}
\begin{th:whence}
  \Cref{le:ce2fin} allows us to focus on one implication: $\cat{FIN}_{(\iota_i)}$ implies coherent embeddability (or, equivalently by \Cref{th:normasound}, algebraic soundness). It will be enough, moreover, to prove a statement formally weaker than coherent embeddability:

  {\bf Claim:} For a fixed $i_0\in I$ and an irreducible unitary representation $\pi\in\widehat{G_{i_0}}$, there are unitary representations $\varphi_{i,\pi}:G_i\to U(V_{\pi})$ on the same finite-dimensional Hilbert space, agreeing on $H$, and such that $\pi\le \pi_{i_0}$.

  Assuming the claim for the moment, we can conclude as follows:
  \begin{itemize}
  \item First consider, for each $i$, the map
    \begin{equation}\label{eq:prodvarphi}
      G_i\xrightarrow{\quad\prod_{\pi}\varphi_{i,\pi}\quad} \prod_{\pi}U(V_{\pi});
    \end{equation}
    these (for differing $i$) agree on $H$ by construction, and at $i_0$ the map in question is an embedding (because the direct sum of the irreducible representations of $G_{i_0}$ is faithful; this follows, for instance, from \cite[Corollary 2.36]{hm4}).
  \item Then, repeat the procedure for each $i$, obtaining a product of maps of the form \Cref{eq:prodvarphi} over $I$. {\it Those} maps will give the desired coherent embedding.
  \end{itemize}
  
  It thus remains to prove the claim, writing `0' for `$i_0$' to avoid overburdening the notation (so $G_{i_0}$ becomes $G_0$, etc.). As in \Cref{eq:pidec}, we have
  \begin{equation*}
    \pi|_H\cong \bigoplus_{\rho\in C_{\pi}}\rho^{\oplus m}\text{ for a Clifford class }C_{\pi}\text{ of $\iota_0$ and some }m\in \bZ_{>0}.
  \end{equation*}  
  Now consider the possibly-larger but still finite (by assumption) Clifford class
  \begin{equation*}
    \widetilde{C}\supseteq C_{\pi}\text{ associated to the entire family }(\iota_i). 
  \end{equation*}
  By the multiplicity-finiteness half of $\cat{FIN}_{(\iota_i)}$, for every $\rho\in\widetilde{C}$ the Clifford multiplicities $m_{\rho|\iota_i}$ range over some finite set of positive integers. We can thus construct a finite-dimensional $H$-representation
  \begin{equation*}
    \bigoplus_{\rho\in \widetilde{C}}\rho^{\oplus M},\quad
    M:=m\cdot\mathrm{lcm}(m_{\rho|\iota_i},\ i\in I). 
  \end{equation*}
  It is
  \begin{itemize}
  \item obtainable as a restriction of a $G_i$-representation for every $i$;
  \item and furthermore, for $G_0$ that representation can be chosen so as to contain the original $\pi\in \widehat{G_0}$.
  \end{itemize}
  Realizing said $G_i$-representations on a common space, and adjusting them by unitary conjugation so they agree on $H$ (since their restrictions to it are by construction isomorphic), we have the sought-after claim.
\end{th:whence}

A number of conclusions follow from \Cref{th:whence}. As \Cref{ex:heis} (and its contrast to finite families of central embeddings: \Cref{re:km}), positive results will generally be stronger for {\it finite} families of normal embeddings.

Throughout, the automorphism group $\Aut(G)$ of $G\in\cat{CGp}$ is topologized in the one sensible fashion (e.g. \cite[paragraph following Lemma 6.62]{hm4}): with the uniform topology on
\begin{equation*}
  \Aut(G)\subset \left(\text{continuous maps }G\to G\right). 
\end{equation*}
Because the domain $G$ of the maps is itself compact, This is also the {\it compact-open} topology of \cite[Definition 43.1]{wil-top} or \cite[Definition preceding Theorem 46.8]{mun}. 

\begin{corollary}\label{cor:fincpctorb}
  A finite family of normal embeddings $\iota_i:H\trianglelefteq G_i$, $i\in I$ in \cat{CGp} is algebraically sound if and only if the images of the morphisms
  \begin{equation}\label{eq:adgi}
    Ad_i:G_i\xrightarrow[]{\quad\text{conjugation action on }H\quad} \Aut(H),\ i\in I
  \end{equation}
  are all contained in a compact subgroup of $\Aut(H)$.  
\end{corollary}
\begin{proof}
  One implication is clear: algebraic soundness implies means that all conjugation actions by the $G_i$ are pulled back from a single action on $H$ by $\coprod_{H,\cat{CGp}}G_i$.

  Conversely, by \Cref{th:normasound,th:whence} algebraic soundness is equivalent to $\cat{FIN}_{(\iota_i)}$, and for finite families the multiplicity-finiteness half of $\cat{FIN}$ is satisfied automatically. Algebraic soundness thus boils down to the Clifford-class-finiteness condition in \Cref{eq:fincond}. Because a compact group acting on $H$ will have finite orbits on $\widehat{H}$, so that finiteness condition follows from the compactness of the closed subgroup of $\Aut(H)$ generated by all $Ad_i(G_i)$.
\end{proof}

A variant, following immediately from \Cref{cor:fincpctorb}:

\begin{corollary}\label{cor:cpctautgp}
  A finite family of $\cat{CGp}$-embeddings $H\trianglelefteq G_i$ is algebraically sound if $\Aut(H)$ is compact.  \qedhere
\end{corollary}

This applies, in particular, to finite $H$. It also shows that in \Cref{ex:notemb} the smallest torus $H$ that could have been used was $\bT^2$: the 1-torus $\bT^1\cong \bS^1$ has finite automorphism group ($\Aut(\bS^1)\cong \bZ/2$).

The discussion also offers some motivation for the following dichotomous classification of compact groups: some always generate algebraically sound finite families of normal embeddings, and some do not.

\begin{corollary}\label{cor:lgcpct}
  For a compact group $H$, the following conditions are equivalent:
  \begin{enumerate}[(a)]
  \item\label{item:14} Every finite family of normal \cat{CGp}-embeddings $\iota_i:H\trianglelefteq G_i$ is algebraically sound.
  \item\label{item:15} The poset of compact subgroups of $\Aut(H)$ is filtered.
  \end{enumerate}
\end{corollary}
\begin{proof}
  This is a fairly straightforward application of \Cref{cor:fincpctorb}: on the one hand, \Cref{item:15} implies that the images of the (finitely many) maps \Cref{eq:adgi} are all contained in a single compact subgroup of $\Aut(H)$. On the other, if $H_i\le \Aut(H)$, $i=0,1$ are two compact subgroups of $\Aut(H)$ not contained in any common compact subgroup, then the pair
  \begin{equation*}
    H\trianglelefteq G_i:=H\rtimes H_i,\ i=0,1
  \end{equation*}
  will be algebraically unsound by \Cref{cor:fincpctorb}. 
\end{proof}

\begin{remarks}\label{res:coenough}
  \begin{enumerate}[(1)]
  \item A small subtlety is involved in topologizing $\Aut(G)$. The discussion following \cite[Lemma 6.62]{hm4} is a little more involved than alluded to above. That source
    \begin{itemize}
    \item first equips $\Aut(G)$ with the compact-open topology, denoted there by $\cO_1$;
    \item and then strengthens that to what is then denoted by $\cO_1\vee \cO_2$, where $\cO_2$ is the pullback of $\cO_1$ through the inverse map. 
    \end{itemize}
    This last step is intended to ensure the continuity of the inverse on $\Aut(G)$ (since the latter ought to be a topological {\it group}), but does not seem to me to be necessary. In other words, it appears that the compact-open topology already makes $\Aut(G)$ into a topological group.

    The (joint) continuity of composition is already noted in \cite[Lemma 6.62]{hm4}, so the only issue is the continuity of the inverse. To check it, consider a uniformly convergent {\it net} \cite[Definition 11.2]{wil-top} $\alpha_i\to \alpha$ in $\Aut(G)$. If $\alpha_i^{-1}$ were {\it not} uniformly convergent to $\alpha^{-1}$, we would have
    \begin{equation*}
      \alpha_i^{-1}(x_i)\in \alpha^{-1}(x_i)F,\quad \text{for some closed }F\subset G\setminus\{1\}
    \end{equation*}
    and a net $(x_i)\in G$. Set
    \begin{equation*}
      y_i:=\alpha_i^{-1}(x_i),\quad z_i:=\alpha^{-1}(x_i). 
    \end{equation*}
    Passing, if needed, to convergent subnets (which nets always have, since the ambient space is compact \cite[Theorems 11.5 and 17.4]{wil-top}), we can assume that $y_i\to y$ and $z_i\to z$. But then
    \begin{equation*}
      \begin{aligned}
        \alpha_i\to \alpha
        \Longrightarrow
        & x_i=\alpha_i(y_i)\to \alpha(y)\text{ and}\\
        & x_i=\alpha(z_i)\to \alpha(z).
      \end{aligned}    
    \end{equation*}
    It follows that $y=z$ (because $\alpha$ is one-to-one), but on the other hand
    \begin{equation*}
      y_i\in z_iF\Longrightarrow y\in zF;
    \end{equation*}
    this is the desired contradiction.
  \item One reason why the stronger topology might have appeared necessary is that $\Aut(G)\subset \End(G)$ is not generally {\it closed} in the uniform topology: see \Cref{ex:conotclosed}.
  \end{enumerate}
\end{remarks}

\begin{example}\label{ex:conotclosed}
  Let $G:=(\bZ/2)^{\bZ_{\ge 0}}$ (the group of binary $\bZ_{\ge 0}$-indexed sequences), and for each positive integer $n$ denote by $\alpha_n\in \Aut(G)$ the automorphism that cycles the first $n$ terms of a sequence leftward and leaves the rest fixed. The sequence converges uniformly to the surjective, 2-to-1 endomorphism
  \begin{equation*}
    (t_n)_{n\in \bZ_{\ge 0}}
    \xmapsto{\quad}
    (t_{n+1})_{n\in \bZ_{\ge 0}}.
  \end{equation*}
  of $G$. 
\end{example}

In general, the equivalent conditions of \Cref{cor:lgcpct} do not imply that $\Aut(H)$ itself is compact:

\begin{example}\label{ex:qhat}
  Take $H:=\widehat{\bQ}$, the Pontryagin dual of the discrete additive group $(\bQ,+)$. We have
  \begin{equation*}
    \Aut(H)\cong \Aut(\bQ,+)\cong\text{ the multiplicative group }(\bQ^{\times},\cdot),
  \end{equation*}
  acting by multiplication. Because $\Aut(H)$ is discrete torsion-free it has only trivial compact subgroups, so $H$ meets the criteria of \Cref{cor:lgcpct}. 
\end{example}

Note, though, that in \Cref{ex:qhat} the automorphism group does have a {\it largest} compact subgroup: the trivial one. Naturally, if $\Aut(H)$ has a largest compact group, then the filtration condition in \Cref{cor:lgcpct} \Cref{item:15} holds. One might then ask whether the converse holds; that it does not can be seen from \Cref{ex:nolgst}.

\begin{example}\label{ex:nolgst}
  \cite[Lemma 2.1]{fourn_elem-2} gives an example of an abelian group $D$ with $\Aut(D)\cong \bZ/2$. If we modify that construction (as we will momentarily) so as to obtain $\Aut(D)\cong (\bZ/2)^{\aleph_0}$ for discrete abelian $D$, then we can set $H:=\widehat{D}$:
  \begin{equation*}
    \Aut(H)\cong \Aut(D)\cong (\bZ/2)^{\aleph_0}\text{ (with the discrete topology)},
  \end{equation*}
  which means that every compact subgroup of $\Aut(H)$ is finite, $\Aut(H)$ is the union of these finite subgroups, but it is not itself compact.

  On to the construction then: recall first the groups in \cite[Lemma 2.1]{fourn_elem-2}: for a set
  \begin{equation*}
    \cP = \{p_i\ |\ 1\le i\le k\}\sqcup \{p_{ij}\ |\ 1\le i<j\le k\}
  \end{equation*}
  of distinct primes and a set $\cS=\{s_i\ |\ 1\le i\le k\}$ of $\bQ$-linearly independent reals, define $D_{\cS,\cP}$ to be the subgroup of $(\bR,+)$ generated by
  \begin{equation*}
    \frac {s_i}{p_i^n},\ 1\le i\le k
    \quad\text{and}\quad
    \frac {s_i+s_j}{p_{ij}^n},\ 1\le i<j\le k
    \quad\text{for}\quad
    n\in \bZ_{>0}.
  \end{equation*}
  To modify that example as desired, take
  \begin{equation*}
    D=\bigoplus_{t=1}^{\infty} D_{\cS_t,\cP_t}
  \end{equation*}
  for mutually disjoint sets $\cS_t$ with $\coprod_{t}\cS_t$ still $\bQ$-linearly independent and mutually disjoint sets $\cP_t$ of primes. A simple adaptation of the argument proving \cite[Lemma 2.1]{fourn_elem-2} then shows that the automorphisms of $D$ are precisely those acting separately as automorphisms of the individual summands $D_{\cS_t,\cP_t}$, so indeed
  \begin{equation*}
    \Aut(D)\cong \prod_{t}\Aut(D_{\cS_t,\cP_t})\cong (\bZ/2)^{\aleph_0}. 
  \end{equation*}
\end{example}

For ``smaller'' groups things are different. Per common practice (e.g. \cite[Definition A1.59]{hm4}), define the {\it rank} of a (discrete) abelian group $D$ by
\begin{equation*}
  \rk(D):=\dim_{\bQ} D\otimes_{\bZ}\bQ.
\end{equation*}


\begin{theorem}\label{th:cpctconn}
  Let $H$ be a compact connected Lie group, $Z:=Z(H)$ its center and $Z_0$ the identity component of that center. Consider the following conditions:
  \begin{enumerate}[(a)]
  \item\label{item:16} $H$ has no central 2-torus.
  \item\label{item:17} The Pontryagin dual $\widehat{Z}$ of the center of $H$ has rank $\le 1$.
  \item\label{item:18} The automorphism group $\Aut(H)$ is compact.
  \item\label{item:20} The automorphism group $\Aut(H)$ has a largest compact subgroup.
  \item\label{item:23} The finite-order automorphisms of $Z_0$ induced by $\Aut(H)$ are precisely the identity and the inversion automorphism $z\mapsto z^{-1}$.
  \item\label{item:19} The equivalent criteria of \Cref{cor:lgcpct}. 
  \end{enumerate}
  We then have the implications
  \begin{equation*}
    \Cref{item:16} \iff \Cref{item:17} \iff \Cref{item:18}
    \qquad
    \Longrightarrow
    \qquad
    \Cref{item:20} \iff \Cref{item:23} \iff \Cref{item:19}.
  \end{equation*}
\end{theorem}

We make a detour for a number of preparatory results. For the following statement, recall \cite[\S I.9.3, D\'efinition 3]{bourb_top-1-4} that a subspace $A\subseteq X$ of a topological space is {\it relatively compact} if it is contained in a compact subspace. Equivalently, since we will be working with Hausdorff spaces, this means that the closure $\overline{A}$ is compact.

\begin{proposition}\label{pr:cpctoncent}
  Let $H$ be a compact Lie group and $Z_0(H_0)\cong \bT^k$ the identity component of the center of its identity component.

  A subgroup of $\Aut(H)$ is relatively compact if and only if its image in
  \begin{equation*}
    \Aut(Z_0(H_0))\cong \Aut(\bT^k)\cong GL(k,\bZ)
  \end{equation*}
  is finite.
\end{proposition}
\begin{proof}
  The restriction morphism
  \begin{equation*}
    \Aut(H)\ni \alpha\mapsto \alpha|_{Z_0(H_0)} \in GL(k,\bZ)
  \end{equation*}
  is continuous, so the forward implication is immediate; we thus focus on the converse.

  \begin{enumerate}[(I)]
  \item {\bf Reducing the problem to connected groups.} Let $\pi_0(H):=H/H_0$ be the group of connected components of $H$. An automorphism of $H$ induces one on both $H_0$  and $\pi_0(H)$, hence a morphism
    \begin{equation}\label{eq:htoh0pi0}
      \Aut(H)\to \Aut(H_0)\times \Aut(\pi_0(H)). 
    \end{equation}
    I claim that the kernel of that morphism is compact, so relative compactness transports back and forth between $\Aut(H)$ and $\Aut(H_0)$ (since $\pi_0(H)$ is finite); this will effect the desired reduction.

    To prove the claim, consider an automorphism $\alpha$ in the kernel of \Cref{eq:htoh0pi0} and set
    \begin{equation}\label{eq:phia}
      \varphi(x):=\alpha(x)x^{-1},\ \forall x\in H. 
    \end{equation}
    The map is constant on each left $H_0$-coset, and hence also on {\it right} $H_0$-cosets. This means that $\varphi$ descends to a map (denoted by the same symbol) $\varphi:\pi_0(H)\to Z(H)$. Furthermore, the fact that $\alpha$ was a group automorphism implies that $\varphi$ is a {\it 1-cocycle} \cite[\S 2.3]{ev-coh} with respect to the action
    \begin{equation*}
      \pi_0(H)\times Z(H) \ni (x,y)\xmapsto[]{\quad} {}^x y\in Z(H)
    \end{equation*}
    of $\pi_0(H)$ on $Z(H)$ (action induced by conjugation):
    \begin{equation*}
      \varphi(xy) = \varphi(x)\cdot {}^x\varphi(y),\ \forall x,y.
    \end{equation*}
    Conversely, every 1-cocycle $\varphi$ gives an automorphism $\alpha$ in the kernel of \Cref{eq:phia} via the same \Cref{eq:phia}. We thus have an identification of that kernel with the space $Z^1(\pi_0(G),Z(H))$ of 1-cocycles, which is plainly compact.

  \item {\bf Connected $H$.} We have \cite[Theorem 9.24]{hm4}
    \begin{equation}\label{eq:hprod}
      H\cong  H^* \big/ \Delta,\quad H^*:= Z_0\times \prod_{i=1}^{\ell}S_i,
    \end{equation}
    where
    \begin{itemize}
    \item $Z_0$ is the identity component of the center $Z(H)$;
    \item $S_i$ are compact, connected, simple and simply-connected Lie groups;
    \item and $\Delta$ is a finite central subgroup of the product in question.
    \end{itemize}
    The product of the $S_i$ is precisely the universal cover of the commutator group $H'$ \cite[Corollary 9.6 and Theorem 9.19]{hm4}, so an automorphism of $H$ also induces one on that product $S:=\prod_i S_i$. $Z_0$ is also, of course, preserved by automorphisms, so
    \begin{equation*}
      \Aut(H)\cong \{\alpha\in \Aut(Z_0)\times \Aut(S)\ |\ \alpha(\Delta)=\Delta\}. 
    \end{equation*}
    The second factor $\Aut(S)$ is always compact \cite[Theorem 6.61]{hm4}, so the relative compactness of a subgroup of $\Aut(H)$ is equivalent to that of its image into $\Aut(Z_0)\cong GL(k,\bZ)$, i.e. to its finiteness.    
  \end{enumerate}
  This finishes the proof. 
\end{proof}

Observe, next, that the phenomenon operative in \Cref{ex:nolgst} does not occur for Lie groups:

\begin{proposition}\label{pr:cpctlie-lgstcpct}
  A compact Lie group $H$ satisfies the equivalent conditions of \Cref{cor:lgcpct} if and only if $\Aut(H)$ has a largest compact subgroup.
\end{proposition}
\begin{proof}
  The backward implication is immediate, so we focus on the converse.
  
  $Z_0:=Z_0(H_0)$ is some torus, say $\bT^k$. By \Cref{pr:cpctoncent}, a closed subgroup of $\Aut(H)$ is compact if and only if its image in
    \begin{equation*}
      \Aut(Z_0)\cong \Aut(\bT^k)\cong \Aut(\bZ^k)\cong GL(k,\bZ)
    \end{equation*}
    is finite (or equivalently, compact).
    
    Being a finitely generated (e.g. \cite[Lemma 3.1]{miln_k}) subgroup of $GL(k,\bC)$, the integer general linear group has a finite-index torsion-free subgroup \cite[Chapter 5, Theorem 4.1]{cass_lf}. Every chain of finite subgroups must thus stabilize, so filtration of the set of compact subgroups of $\Aut(H)$ implies that that set must have a largest element (`largest' in the sense of inclusion).
\end{proof}

Per \Cref{pr:cpctoncent}, the non-compactness of $\Aut(H)$ is all still ``visible'' after an application of the restriction
\begin{equation}\label{eq:restoz0}
  \rho:\Aut(H)\to \Aut(Z_0(H_0)). 
\end{equation}
This observation extends to the property of having a largest compact subgroup, of interest here because of \Cref{pr:cpctlie-lgstcpct}. 

\begin{corollary}\label{cor:haslgst-hasfin}
  Let $H$ be a compact Lie group and set $Z_0:=Z_0(H_0)$.

  $\Aut(H)$ has a largest compact subgroup if and only if its image $\rho(\Aut(H)) \le \Aut(Z_0)$ through the map \Cref{eq:restoz0} has a largest finite subgroup.
\end{corollary}
\begin{proof}
  The proof of \Cref{pr:cpctoncent} shows that the kernel of \Cref{eq:restoz0} is compact. It is of course also normal, so there is an isomorphism between the (inclusion-ordered) poset of compact subgroups of $\Aut(H)$ and that of compact (hence finite) subgroups of
  \begin{equation*}
    \rho(\Aut(H))\le \Aut(Z_0)\cong GL(k,\bZ),\quad k:=\dim Z_0. 
  \end{equation*}
\end{proof}

\begin{proposition}\label{pr:justzz1}
  Let $H$ be a compact connected Lie group, $Z_0:=Z_0(H)$, and $\rho:\Aut(H)\to \Aut(Z_0)$ the restriction map \Cref{eq:restoz0}.

  $\Aut(H)$ has a largest compact subgroup if and only if the torsion of $\rho(\Aut(H))$ is contained in the order-$(\le 2)$ group generated by
  \begin{equation}\label{eq:zz1}
    Z_0\ni z\mapsto z^{-1}\in Z_0. 
  \end{equation}
\end{proposition}
\begin{proof}
  According to \Cref{cor:haslgst-hasfin}, $\Aut(H)$ has a largest compact subgroup precisely when its restriction $\rho(\Aut(H))$ to $Z_0$ has a largest finite subgroup. In turn, this is equivalent to the requirement that all finite-order elements of $\rho(\Aut(H))$ constitute a subgroup.

  In particular, the backward implication ($\Longleftarrow$) is obvious, so we focus on the converse. Furthermore, $Z_0$ is some torus $\bT^k$; if $k\le 1$ there is nothing to prove (since the only non-trivial automorphism of a circle is \Cref{eq:zz1}), so we assume $k\ge 2$ throughout.

  
   Recall the decomposition \Cref{eq:hprod}:
  \begin{equation*}
    H= H^*/\Delta,\quad H^*:= Z_0\times \prod_{i=1}^{\ell}S_i,
  \end{equation*}
  where $S:=\prod S_i$ is the universal cover of the commutator group $H'$ of $H$. What is more, the identification $\Delta\cong Z_0\cap H'$ of \cite[Theorem 9.24]{hm4} implies that $\Delta\le H^*=Z_0\times S$ is the graph
  \begin{equation*}
    \{(\varphi z,z)\ |\ z\in \Delta_S\le Z(S)\}
    \quad\text{of a surjection}\quad
    \Delta_S\to Z_0\cap H'
  \end{equation*}
  of finite abelian groups, where $\Delta_S\le Z(S)$ is some subgroup of the center of $S$. In particular, the left-hand components $\varphi(z)$ of $\Delta$ precisely make up the finite group $\Delta_0:=Z_0\cap H'$.
  
  Now, because $\Delta_0$ is a finite subgroup of the $k$-torus $Z_0$, it must be contained in the $d$-torsion subgroup $Z_0[d]<Z_0$ for some $d$. Denote by
  \begin{equation*}
    T_d < GL(k,\bZ)\cong \Aut(Z_0)\cong \Aut(\bT^k) 
  \end{equation*}
  the group of upper-triangular {\it unipotent} matrices (i.e. diagonal entries 1) whose off-diagonal entries are divisible by $d$. Any $\alpha_0\in T$, regarded as an automorphism of $Z_0$, will induce an automorphism $(\alpha_0,\id)$ of $H^*\cong Z_0\times S$ that fixes $\Delta$ pointwise (because $\alpha_0$ fixes $\Delta_0$ pointwise). It follows that $(\alpha_0,\id)$ descends to an
  \begin{equation*}
    \alpha\in \Aut(H) = \Aut(H^*/\Delta)
    \quad\text{with}\quad
    \rho(\alpha)=\alpha_0. 
  \end{equation*}
  In short, we now have a subgroup $\widetilde{T}_d\le \Aut(H)$, consisting of the lifts $\alpha$ for $\alpha\in T_d$, as just constructed, mapped by $\rho$ isomorphically onto $T_d<GL(k,\bZ)$:
  \begin{equation*}
    \Aut(H) > \widetilde{T}_d\xrightarrow[\cong]{\quad\rho\quad} T_d<GL(k,\bZ).
  \end{equation*}
  If $\rho(\Aut(H))$ has a largest finite subgroup then that subgroup must of course be normal, and hence normalized by $\widetilde{T}_d$. It will then also be centralized by the analogous group $\widetilde{T}_{d'}$ by $d'\in \bZ_{>0}$ sufficiently large in the sense of divisibility (i.e. divisible by $d$ and by sufficiently high powers of sufficiently many primes). To conclude, observe that the centralizer of $T_{d'}$ in $GL(k,\bZ)$ (for any $d'$) consists precisely of the scalar matrices $\pm 1$.
\end{proof}

There are two obvious ways to (try to) extend \Cref{pr:justzz1} to disconnected $H$: $Z_0$ could be either $Z_0(H_0)$ or $Z_0(H)$. Neither version goes through.

\begin{example}\label{ex:torusorder3}
  Let $H:=\bT^2\rtimes \bZ/3$, with the order-3 automorphism
  \begin{equation*}
    \alpha\in GL(2,\bZ)\cong \Aut(\bZ^2)\cong \Aut(\bT^2)
  \end{equation*}
  attached to a generator of $\bZ/3$ given by the matrix $
  \begin{pmatrix}
    \phantom{-}0&\phantom{-}1\\
    -1&-1
  \end{pmatrix}
  $. $\Aut(H)$ has a finite-index subgroup consisting of automorphisms fixing $\bZ/3$ pointwise and identifiable with the centralizer of $\alpha$ in $\Aut(\bT^2)$. That centralizer is the order-6 cyclic group generated by $-\alpha$, so $\Aut(H)$ itself is finite. Automorphisms of $H$ restrict to automorphisms of $Z_0(H_0)=\bT^2$ other than \Cref{eq:zz1}, however.
  
  This shows that the requirement of \Cref{pr:justzz1} with $Z_0:=Z_0(H_0)$ is not generally necessary, for disconnected $H$, in order to have compact $\Aut(H)$.
\end{example}

On the other hand, meeting the requirement of \Cref{pr:justzz1} with $Z_0:=Z_0(H)$ is typically not {\it sufficient}, as the following variant of \Cref{ex:torusorder3} shows. 

\begin{example}\label{ex:torus2order3}
  This time set $H:=\bT^4\rtimes \bZ^3$, with a generating automorphism $\alpha$ operating as a block-diagonal matrix, with two blocks identical to the $2\times 2$ matrix of \Cref{ex:torusorder3}:
  \begin{equation*}
    \alpha=
    \begin{pmatrix}
      \phantom{-}0&\phantom{-}1&\phantom{-}0&\phantom{-}0\\
      -1&-1&\phantom{-}0&\phantom{-}0\\
      \phantom{-}0&\phantom{-}0&\phantom{-}0&\phantom{-}1\\
      \phantom{-}0&\phantom{-}0&-1&-1
    \end{pmatrix}
    \in GL(4,\bZ)\cong \Aut(\bT^4). 
  \end{equation*}
  The center of $H$ consists of the $\alpha$-invariants in $\bT^4$, and is thus trivial: $Z(H)=\{1\}$. So, then, is $Z_0(H)$, and the requirement of \Cref{pr:justzz1} is satisfied trivially. On the other hand, the $4\times 4$ matrices 
  \begin{equation*}
    \begin{pmatrix}
      -I&0\\
      \phantom{-}0&I
    \end{pmatrix}
    \quad\text{and}\quad
    \begin{pmatrix}
      I&I\\
      0&I
    \end{pmatrix}
  \end{equation*}
  (with the symbols depicting $2\times 2$ blocks) both commute with $\alpha$ and hence induce automorphisms of $H$. Conjugating the former repeatedly with the latter shows that there are infinitely many distinct order-2 automorphisms of $H$ (so no largest compact subgroup of $\Aut(H)$). 
\end{example}

\begin{example}\label{ex:lgstcpct-notcpct}
  We will construct a compact connected Lie group $H$ as a quotient $H^*/\Delta$ of the form \Cref{eq:hprod}, where
  \begin{itemize}
  \item $Z_0\cong \bT^2$;
  \item the semisimple factor $S=S_1\times S_2$ will be of the form $SU(3^k)\times SU(3^l)$ for $k>l\ge 2$ (so the two $S_i$ are the two special unitary groups in question);
  \item and $\Delta$ will be the graph of a morphism
    \begin{equation*}
      Z(S)\cong Z(S_1)\times Z(S_2) \cong \bZ/3^k\times \bZ/3^l\xrightarrow[]{\quad\varphi\quad}\Delta_0<\bT^2
    \end{equation*}
    which identifies the order-3 subgroups of the two factors $\bZ/3^k$ and $\bZ/3^l$ (so the kernel of $\varphi$ is a copy of $\bZ/3$ embedded anti-diagonally in the product $\bZ/3^k\times \bZ/3^l$).
  \end{itemize}
  The automorphism group of $SU(d)$, $d\ge 3$ is $PSU(d)\rtimes (\bZ/2)$ with $\bZ/2$ acting as complex conjugation \cite[\S IX.5, Theorem 5]{jc}. Automorphisms of $S_i\cong SU(\bullet)$ will thus act either trivially on the centers of $S_i$ or by inversion, and the identification of the two $\bZ/3$ copies in those centers ensures that every element of $\Aut(H)$ either inverts all of $\Delta$ or fixes it pointwise.

  I claim that the only finite-order automorphisms of $\bT^2$ induced from automorphisms of $H$ are the identity and $z\mapsto z^{-1}$, so that $\Aut(H)$ has a largest compact subgroup by \Cref{pr:justzz1}.
  
  To check the claim, Consider an automorphism $\alpha\in \Aut(H)$ inducing a finite-order automorphism
  \begin{equation*}
    \rho \alpha\in GL(2,\bZ)\cong \Aut(\bT^2). 
  \end{equation*}
  We have a pushout decomposition
  \begin{equation*}
    SL(2,\bZ)\cong \bZ/6\coprod_{\bZ/2}\bZ/4
  \end{equation*}
  with
  \begin{equation*}
    \bZ/6\text{ generated by }
    \begin{pmatrix}
      0&-1\\
      1&\phantom{-}1
    \end{pmatrix}
    ,\quad
    \bZ/4\text{ generated by }
    \begin{pmatrix}
      0&-1\\
      1&\phantom{-}0
    \end{pmatrix}
  \end{equation*} 
  \cite[equation (3) and attached footnote 2 following Proposition 1.18]{Hain11}, so any finite subgroup of $SL(2,\bZ)$ can be conjugated inside either $\bZ/6$ or $\bZ/4$ \cite[\S I.4.3, Theorem 8]{ser_tr}. The square $\rho\alpha^2$ is a finite-order element of $SL(2,\bZ)$, so it can be so conjugated. But the $GL(2,\bZ)$-action
  \begin{equation*}
    \bT^2\ni (z,w)
    \xmapsto[]{\quad
      \begin{pmatrix}
        a&b\\
        c&d
      \end{pmatrix}
      \quad}
    (z^a w^b,\ z^c w^d)
  \end{equation*}
  on the torus is such that no non-trivial element of either $\bZ/6$ or $\bZ/4$ fixes (pointwise) a subgroup of $\bT^2$ as large as
  \begin{equation}\label{eq:imphi}
    \mathrm{im}(\varphi)\cong \bZ/3^k\times \bZ/3^{l-1},
  \end{equation}
  so $\varphi\alpha^2=1$. It follows, then, that the finite-order $\varphi\alpha$ are all involutions. Furthermore, the same argument shows that the only $\varphi\alpha$ in $SL(2,\bZ)$ are the identity and $z\mapsto z^{-1}$, so it remains to argue that no finite-order $\varphi\alpha$ have determinant $-1$ (i.e. eigenvalues $\pm 1$, each with multiplicity 1). 

  Indeed, every $\bZ/2$-action on $\bT^2$ with eigenvalues $\pm 1$ is conjugate to either
  \begin{equation*}
    \begin{pmatrix}
      -1&0\\
      \phantom{-}0&1
    \end{pmatrix}
    \quad\text{or}\quad
    \begin{pmatrix}
      -1&1\\
      \phantom{-}0&1
    \end{pmatrix}
  \end{equation*}
  \cite[Theorem 74.3]{cr_rep}. Either way, such a group fixes (pointwise) a direct summand $\bS^1< \bT^2$, whereas \Cref{eq:imphi} is not contained in a circle. 
\end{example}

\pf{th:cpctconn}
\begin{th:cpctconn}
  \begin{enumerate}[]
  \item {\bf \Cref{item:16} $\Longleftrightarrow$ \Cref{item:17}} This is basic structure theory: the center is a compact abelian Lie group and hence a torus $\bT^k$ times a finite group $F$ \cite[Corollary 4.2.6]{de}, so its dual isomorphic to $\bZ^d\cong \widehat{F}$. The rank $d$ of $\widehat{Z(H)}$, then, is the dimension of the largest central torus.

  \item {\bf \Cref{item:17} $\Longrightarrow$ \Cref{item:18}} Reprise the setup and notation of \Cref{eq:hprod}, in the proof of \Cref{pr:cpctoncent}. We can then conclude similarly: $\Aut(H)$ as a whole is compact if $Z_0$ is either trivial or a circle, as the current branch of the proof assumes.

  \item {\bf \Cref{item:18} $\Longrightarrow$ \Cref{item:16}} We saw in the proof of \Cref{pr:justzz1} that if $Z_0$ is at least 2-dimensional then $\Aut(H)$ contains infinite torsion-free subgroups (denoted there by $\widetilde{T}_d$).

    This settles the mutual equivalence of \Cref{item:16}, \Cref{item:17} and \Cref{item:18}.
    
  \item {\bf \Cref{item:18} $\Longrightarrow$ \Cref{item:20}} Obvious, leaving the mutual equivalence of \Cref{item:20}, \Cref{item:23} and \Cref{item:19}.
    
  \item {\bf \Cref{item:20} $\Longleftrightarrow$ \Cref{item:19}} This follows from \Cref{pr:cpctlie-lgstcpct} (which does not need connectedness).

  \item {\bf \Cref{item:23} $\Longrightarrow$ \Cref{item:20}} A consequence of \Cref{pr:justzz1}.

  \item {\bf \Cref{item:20} $\Longrightarrow$ \Cref{item:23}} This too will follow from \Cref{pr:justzz1}, once we show that the inversion automorphism $z\mapsto z^{-1}$ is in fact achievable by restricting an automorphism of $H$.

    To see this, recall the realization \Cref{eq:hprod} of $H$ as a quotient, with $\Delta$ the graph of a surjective morphism
    \begin{equation*}
      \Delta_S\xrightarrow[]{\quad\varphi\quad} Z_0\cap H'
    \end{equation*}
    defined on a central subgroup $\Delta_S$ of $S:=\prod_i S_i$. If there were an automorphism of $S$ operating as inversion on $Z(S)$, we would be done: $\varphi$ would intertwine it and inversion on $Z_0$, so everything in sight would descend to an automorphism of $H=H^*/\Delta$.

    Since $S$ is a product of simply-connected, simple compact Lie groups, it is enough to check that these have automorphisms (finite-order, in fact) acting as inversion on their centers. This can be done on a case-by-case basis, using the classification of such compact Lie groups (e.g. \cite[Chapter 3, table following Corollary 5]{stein_lec}, which also conveniently lists, in the second column, the centers of the respective groups):
    \begin{itemize}
    \item For types where the center is either trivial or an elementary abelian 2-group (i.e. a product of factors $\bZ/2$), the claim is self-evident. This is the case for $B_l$, $C_l$, $D_{2n}$, $E_7$, $E_8$, $F_4$ and $G_2$.
    \item For type $A_{l}$, $S_i$ would be the special unitary group $SU(l+1)$, with center $\bZ/(l+1)$. It does have an automorphism inverting the center, namely that sending a matrix to its complex conjugate (equivalently: to its inverse transpose).
    \item In type $D_{2n+1}$ the center is $\bZ/4$, and the compact group is the {\it spin group} $Spin_{4n+2}$. That it has an automorphism (necessarily outer) inverting its center is verified in \cite[discussion surrounding equation (8.20), pp.72-73]{bince_lie}.
      
    \item Finally, in type $E_6$ with center $\bZ/3$ an outer automorphism again inverts the center, for instance by \cite[Theorem 3.7.1]{yok_exc} (which computes the invariant subgroup of such an automorphism to be of type $F_4$, and hence center-less).
    \end{itemize}

  \end{enumerate}
  The proof is complete.
\end{th:cpctconn}

\subsection{Families of split embeddings}\label{subse:split}

There are (at least) two categories of interest on a group-theoretic multi-pushout $G:=\coprod_{H,\cat{Gp}}G_i$:
\begin{itemize}
\item the finest group topology making the maps $G_i\to G$ continuous, which we will refer to as the {\it final (group) topology} by analogy to, say \cite[\S I.2.4]{bourb_top-1-4};
\item the topology inherited from the map into the Bohr compactification: henceforth the {\it Bohr topology}.
\end{itemize}

\begin{remark}\label{re:topord}
  The two topologies mentioned above will generally not coincide. Suppose $H$ is trivial, for instance, and $G_i$ are metric with at least one infinite (equivalently: non-discrete) and at least one other non-trivial. \cite[Theorem 4.4]{ord_free-0} then shows that the final topology is strictly finer that the topology of \cite[I, \S 3]{ord_free-1-2}, constructed by means of pseudometrics topologizing the $G_i$. We will show below (\Cref{pr:taufiner}), in turn, that the latter is at least as fine as the Bohr topology.
\end{remark}

For a family of compact groups $G_i$, $i\in I$, we recall briefly the topology $\tau$ constructed in \cite[I, \S 3]{ord_free-1-2}.

\begin{construction}\label{con:tautop}
  First, the $G_i$ being compact, their topologies are induced by families of bi-invariant {\it pseudometrics} \cite[Chapter 6, Problem O]{kel_top}. Having chosen a pseudometric $d_i$ on each $G_i$, we can then equip the group-theoretic coproduct $\coprod_{\cat{Gp}}G_i$ with its own (again bi-invariant \cite[I, Lemma 2]{ord_free-1-2}) pseudometric defined uniquely, for a word
  \begin{equation*}
    g_1 g_2\cdots g_n\in \coprod_{\cat{Gp}}G_i,\quad 1\ne g_k\in G_{i_k},\quad i_k\ne i_{k+1},
  \end{equation*}
  by
  \begin{equation}\label{eq:dfromdi}
    d(g_1 g_2\cdots g_n,\ 1) = \inf_{(e_k)}\sum_{k=1}^n d_{i_k}(g_k,e_k),
  \end{equation}
  where the infimum runs over all tuples $e_k\in G_{i_k}$ with
  \begin{equation*}
    e_1 e_2\cdots e_k = 1 \in \coprod_{\cat{Gp}}G_i. 
  \end{equation*}
  The topology $\tau$ is the one induced by all of these pseudometrics $d$, as the $d_i$ vary.
\end{construction}

We can now justify the claim made in passing in \Cref{re:topord}:

\begin{proposition}\label{pr:taufiner}
  For any family of compact groups $G_i$, $i\in I$, the topology $\tau$ on the algebraic coproduct $G:=\coprod_{\cat{Gp}}G_i$ described above is at least as fine as the Bohr topology.
\end{proposition}
\begin{proof}
  We have to argue that if a net
  \begin{equation*}
    (g_{\alpha} = g_{\alpha,1} g_{\alpha,2} \cdots g_{\alpha,n_{\alpha}})_{\alpha} \subset G,\ g_{\alpha,k}\in G_{i_{\alpha,k}}
  \end{equation*}
  approaches $1$ with respect to every pseudometric $d$ on $G$ induced by individual pseudometrics $d_i$ on $G_i$, then it also approaches 1 in every finite-dimensional unitary representation of $G$.

  Consider, then, a unitary representation $\rho:G\to U(V)$ (unitary group on the finite-dimensional Hilbert space $V$), with its restrictions $\rho_i:=\rho|_{G_i}$. It induces pseudometrics $d_i$ on $G_i$, for instance by pulling back the usual metric of $U(V)$, induced by the norm on the finite-dimensional $C^*$-algebra $\End(V)$.
  
  In turn, the $d_i$ induce a pseudometric $d$ on $G$; I claim that $\rho$ is continuous for the pseudometric $d$. Indeed, consider
  \begin{equation*}
    g=g_1\cdots g_n,\quad g'=g'_1\cdots g'_n,\quad g_k,g'_k\in G_{i_k},\quad \sum_k d_{i_k}(g_k,g'_k)\text{ small}.
  \end{equation*}
  We can then write
  \begin{equation*}
    \rho(g'_k) = \rho(g_k)+\varepsilon_k T_k,\quad
    T_k\in \End(V)\text{ of norm }\le 1,\quad
    \sum_k \varepsilon_k\text{ small}. 
  \end{equation*}
  The two products we are comparing are
  \begin{equation*}
    \rho_1(g_1)\cdots \rho_n(g_n)
    \quad\text{and}\quad
    (\rho_1(g_1)+\varepsilon_1 T_1)\cdots (\rho_n(g_n)+\varepsilon_n T_n),
  \end{equation*}
  and the conclusion follows that the difference is dominated in norm by
  \begin{equation*}
    \delta+\delta^2+\delta^3+\cdots = \frac{\delta}{1-\delta},\quad
    \delta:=\sum_{k=1}^n \varepsilon_k. 
  \end{equation*} 
\end{proof}

\begin{remark}\label{re:taustronger}
  In general, the topology $\tau$ of \Cref{con:tautop} will be {\it strictly} finer than the Bohr topology. Indeed, if the $G_i$ are finite, then $\tau$ will be discrete. On the other hand, assuming at least two $G_i$ are non-trivial, the coproduct $G:=\coprod_{\cat{Gp}}G_i$ cannot be discrete in its Bohr topology:

  There is some infinite-order element $g\in G$ (e.g. the product of two non-trivial elements in two of the $G_i$: \cite[Theorem 11.69]{rot-gp}), and the sequence $(g^n)_n$, homeomorphic to $\bZ$ in the $\tau$ topology, will have a Bohr-convergent subsequence $(g^{n_k})_k$ \cite[Theorem 17.4]{wil-top} (to some element of its Bohr compactification), and hence also a subsequence $(g^{n_{k+1}-n_k})_k$ Bohr-convergent to 1.
\end{remark}

Consider a compact group $H$, acting continuously on compact groups $H_i$. It is a natural question (and one related to \Cref{res:hulan} \Cref{item:9} above, as will soon become apparent), whether the resulting action
\begin{equation*}
  H\times \coprod_{\cat{Gp}}H_i\to \coprod_{\cat{Gp}}H_i
\end{equation*}
is continuous for the Bohr topology; this is not the case, in general:

\begin{example}\label{ex:splitnogood}
  Let $H$ be the product $\prod\bZ/p$ over all primes $p$, and equip, for every $p$, the $p$-dimensional torus $H_p:=\bT^p$ with the action of $H$ whereby the index-$p$ factor acts on the Pontryagin dual
  \begin{equation}\label{eq:zptp}
    \bZ^p\cong \widehat{\bT^p}=\widehat{H_p}
  \end{equation}
  via the permutation matrix corresponding to a $p$-cycle and the other factors act trivially.

  Now, suppose $H$ {\it does} act continuously on the Bohr-topologized coproduct $F=\coprod_{\cat{Gp}}H_p$, and hence also on the compact-group coproduct $\coprod_{\cat{CGp}} H_p$. Every unitary representation of $F$ would extend to the pushout \Cref{eq:pushoverh} by \Cref{pr:iffactcont} below.

  To reach a contradiction, consider a character $\chi$ of $\coprod_{\cat{CGp}} H_p$ which restricts, on every $H_p$, to a character whose orbit under $H$ has full size $p$. Naturally, {\it this} representation will not extend to a finite-dimensional one on all of \Cref{eq:pushoverh}: such a representation would have to contain all $p$ of the $H$-iterates of every $\chi|_{H_p}$, for every $p$.
\end{example}

As for the relevance of \Cref{ex:splitnogood} (and the question it answers negatively) to \Cref{res:hulan} \Cref{item:9}, it stems from the following result.

\begin{proposition}\label{pr:iffactcont}
  Let $H$ be a compact group acting continuously on compact groups $H_i$, $i\in I$. The following conditions are equivalent:

  \begin{enumerate}[(a)]

  \item\label{item:12} The resulting action of $H$ on the algebraic coproduct $\coprod_{\cat{Gp}}H_i$ is continuous for some group topology intermediate between the final and Bohr topologies.

  \item\label{item:13} We have a decomposition
    \begin{equation}\label{eq:pushoverh}
      \coprod_{H,\cat{CGp}} (H_i\rtimes H)\cong \left(\coprod_{\cat{CGp}}H_i\right)\rtimes H.  
    \end{equation}
    via the obvious maps.

  \item\label{item:8} The action of $H$ on the algebraic coproduct $\coprod_{\cat{Gp}}H_i$ is continuous for the Bohr topology.

  \end{enumerate}
  Moreover, they imply the algebraic soundness of the family of embeddings $H\le H_i\rtimes H$.

\end{proposition}
\begin{proof}
  We have a decomposition
  \begin{equation*}
    \coprod_{\cat{Gp}}(H_i\rtimes H)\cong \left(\coprod_{\cat{Gp}} H_i\right)\rtimes H.
  \end{equation*}
  If the action in the rightmost term were continuous for some ``intermediate topology'', as in \Cref{item:12}, \Cref{eq:pushoverh} would follow from the respective universality properties of the two sides. This proves that \Cref{item:12} implies \Cref{item:13}, with
  \begin{equation*}
    \text{\Cref{item:13}}
    \Longrightarrow
    \text{\Cref{item:8}}
    \Longrightarrow
    \text{\Cref{item:12}}
  \end{equation*}
  being straightforward.
  
  As for the last statement (regarding algebraic soundness), it is an immediate consequence of \Cref{eq:pushoverh} and the already-cited \cite[Proposition 1]{hul_map} of Hulanicki's to the effect that families of trivial embeddings in \cat{CGp} are algebraically sound.
\end{proof}

All of this, moreover, happens frequently enough to be of some use:

\begin{theorem}\label{th:whenextcont}
  The equivalent conditions of \Cref{pr:iffactcont} hold either
  \begin{enumerate}[(a)]
  \item\label{item:21} when $H$ is finite and $(H_i)$ is an arbitrary family of compact groups equipped with $H$-actions;
  \item\label{item:22} or when $H$ is arbitrary and $(H_i)$ is finite.
  \end{enumerate}  
\end{theorem}
\begin{proof}
  Part \Cref{item:21} is clear: actions of discrete groups are always continuous. As for \Cref{item:22}, I claim that for finite families the $H$-action on the group-theoretic coproduct $\coprod_{\cat{Gp}}H_i$ is continuous for the topology $\tau$ of \Cref{con:tautop} (whence the conclusion, by \Cref{pr:taufiner}).

  Indeed, for every choice of respective pseudometrics $d_i$ on $H_i$, $H$ acts on the $H_i$ {\it uniformly} $d_i$-continuously, because the family is finite: for any $\varepsilon>0$, there is a neighborhood $U\ni 1\in H$ such that $d_i(gx,x)<\varepsilon$ whenever $g\in U$ and $x\in H_i$. The point here is that a single $U$ works for all $i$. But then the $H$-action on $\coprod_{\cat{Gp}}H_i$ is continuous for the induced pseudometric $d$ on that coproduct, as is easily seen from the definition \Cref{eq:dfromdi} of $d$.
\end{proof}

\begin{theorem}\label{th:splitasnd}
  Any family of split embeddings $\iota_i:H\le G_i$ of compact groups is algebraically sound.
\end{theorem}
\begin{proof}
  The splittings provide decompositions
  \begin{equation*}
    G_i\cong H_i\rtimes H.
  \end{equation*}
  Now, {\it finite} subfamilies of $(\iota_i)$ meet the requirements of \Cref{pr:iffactcont} by \Cref{th:whenextcont}, and we can conclude by \Cref{le:famsubfam} \Cref{item:subfam-fam} once we observe that {\it its} hypothesis holds for split embeddings: for a subfamily $(\iota_j)_{j\in J}\subseteq (\iota_i)_{i\in I}$ the map
  \begin{equation*}
    \coprod_{H,\cat{Gp}}G_j = \coprod_{H,\cat{Gp}}H_j\rtimes H
    \to
    \coprod_{H,\cat{Gp}}H_i\rtimes H = \coprod_{H,\cat{Gp}}G_i
  \end{equation*}
  has a left inverse, acting as the identity on each $H_j$ and on the common embedded copy $H$ and as the trivial morphism on $H_i$, $i\not\in J$. 
\end{proof}

\begin{remark}
  The earlier \Cref{ex:notemb} shows that in general, a diagram
  \begin{equation*}
    \begin{tikzpicture}[auto,baseline=(current  bounding  box.center)]
      \path[anchor=base] 
      (0,0) node (l) {$H$}
      +(2,.5) node (u) {$H_0$}
      +(2,-.5) node (d) {$H_1$}
      +(4,.5) node (ur) {$G_0$}
      +(4,-.5) node (dr) {$G_1$}
      ;
      \draw[->] (l) to[bend left=6] node[pos=.5,auto] {$\scriptstyle $} (u);
      \draw[->] (l) to[bend right=6] node[pos=.5,auto,swap] {$\scriptstyle $} (d);
      \draw[->] (u) to[bend left=6] node[pos=.5,auto] {$\scriptstyle $} (ur);
      \draw[->] (d) to[bend right=6] node[pos=.5,auto,swap] {$\scriptstyle $} (dr);
    \end{tikzpicture}
  \end{equation*}
  of embeddings in \cat{CGp} does not produce an embedding $H_0\coprod_H H_1\to G_0\coprod_H G_1$: simply take $H_0=H_1=H=\bT^2$ and $G_i$ as in \Cref{ex:notemb}.

  On the other hand, even for plain coproducts (i.e. multi-pushouts over the trivial group), \Cref{ex:splitnogood} can be leveraged to produce an instance of this misbehavior: see \Cref{ex:coprodnotemb}. Compare once more with \cite[Theorem 4.2]{ped-aut}, which implies that coproducts of $C^*$-algebra embeddings are embeddings (and of course, the familiar pure group-theoretic analogue, which follows from \cite[Theorem 11.66]{rot-gp}, say).
\end{remark}

\begin{example}\label{ex:coprodnotemb}
  For primes $p$ let $H_p=\bT^p$, the $p$-torus, as in \Cref{ex:splitnogood}, and consider embeddings
  \begin{equation*}
    H_p\subset G_p:=H_p\rtimes (\bZ/p),
  \end{equation*}
  with the action of $\bZ/p$ once more as in \Cref{ex:splitnogood} (on the Pontryagin dual \Cref{eq:zptp} by an order-$p$-cycle permutation matrix). The argument for the non-injectivity of
  \begin{equation*}
    \coprod_{\cat{CGp}}H_p\to \coprod_{\cat{CGp}} G_p
  \end{equation*}
  is exactly the same as in \Cref{ex:splitnogood}: a character $\chi$ of the left-hand side that restricts to characters $\chi|_{H_p}$ with size-$p$ orbits under the respective actions of $\bZ/p$ cannot be a restriction of a finite-dimensional unitary representation of $\coprod_{\cat{CGp}} G_p$, because such a representation would absurdly have to contain the entire $(\bZ/p)$-orbit of each $\chi|_{H_p}$. 
\end{example}

The infinitude of the family of embeddings $H_p\le G_p$ is crucial in \Cref{ex:coprodnotemb}:

\begin{proposition}\label{pr:coprodembcgp}
  The coproduct in \cat{CGp} of a finite family of embeddings $H_i\le G_i$ is an embedding.
\end{proposition}
\begin{proof}
  Induction reduces this to a pair of embeddings, say $H_i\le G_i$ for $i=0,1$, and since we can apply the argument symmetrically, once for each of the two indices, it is enough to argue that
  \begin{equation*}
    H\le G\Longrightarrow H\coprod_{\cat{CGp}}K\le G\coprod_{\cat{CGp}}K
  \end{equation*}
  for compact groups $H$, $G$ and $K$.

  Because the Bohr topology on the group-theoretic coproduct $H\coprod_{\cat{Gp}}K$ is induced by that group's finite-dimensional unitary representations, it is enough to argue that any such representation, say $\rho:H\coprod_{\cat{Gp}}K\to U(V)$ (unitary group on a finite-dimensional complex Hilbert space), extends to a unitary representation of $G\coprod_{\cat{Gp}}K$ on a possibly larger space.

  $H\le G$ being an embedding of compact groups, we can certainly extend $\rho|_{H}$ in such a fashion, to $\rho':G\to U(W)$. Now let $K$ act as before (via $\rho$) on $V$ and trivially on its orthogonal complement $V^{\perp}\le W$ to complete the extension of $\rho$ to all of $G\coprod_{\cat{Gp}}K$.
\end{proof}

\begin{remark}\label{re:filteremb}
  \cite[Theorem 4.2]{ped-aut}, mentioned in the discussion preceding \Cref{pr:coprodembcgp} also only handles {\it finite} pushouts of $C^*$-algebras. The reason why this is enough, in the category $\cC^*_1$ of unital $C^*$-algebras, to ensure the injectivity of {\it arbitrary} coproducts of embeddings is that in $\cC^*_1$ the structure maps
  \begin{equation*}
    A_j\to \varinjlim_{i}A_i
  \end{equation*}
  into a filtered colimit of embeddings are again embeddings. This is manifestly not the case in \cat{CGp}: consider the chain
  \begin{equation*}
    SU(2)\subset SU(3)\subset\cdots\subset SU(n)\subset\cdots
  \end{equation*}
  of upper-left-hand-corner embeddings of special unitary groups, whose colimit in \cat{CGp} is trivial.
\end{remark}

\addcontentsline{toc}{section}{References}

\Addresses

\end{document}